\documentclass[11pt]{article}
\usepackage{amsmath, amsfonts, amsthm, amssymb,sectsty,latexsym,dsfont,textcomp}
\usepackage{graphics}
\usepackage{fullpage}
\usepackage{color}
\usepackage{epsfig,psfrag}
\usepackage[]{pstricks,pst-node,pst-text,pst-3d}
\usepackage{courier}

\newtheorem{theorem}{Theorem}[section]
\newtheorem{proposition}[theorem]{Proposition}

\newtheorem{lemma}[theorem]{Lemma}
\newtheorem{remark}[theorem]{Remark}

\def\del{\partial}
\newcommand{\T}{\top}
\newcommand{\Id}{{\bf I}}

\newcommand{\RR}{\mathbb{R}}
\newcommand{\eps}{\varepsilon}

\newcommand{\cD}{\mathcal{D}}
\newcommand{\mddplus}[1]{M^{#1\times #1}_+}
\newcommand{\SP}{\hspace{1pt}}

\numberwithin{equation}{section}

\newcommand{\avar}{\dot{\varphi}}
\newcommand{\bvar}{\frac{\varphi}{s}}
\newcommand{\bvart}{\tfrac{\varphi}{s}}
\newcommand{\bvarp}[1]{\frac{\varphi(#1)}{#1}}

\title{On the construction and properties of weak solutions describing dynamic cavitation
       \thanks{Research partially supported by the EU FP7-REGPOT project "Archimedes Center for
Modeling, Analysis and Computation", the "Aristeia"  program of the Greek Secretariat of Research,
and the EU EST-project "Differential Equations and Applications in Science and
Engineering". Part of this work was completed at the Institute of Applied and Computational Mathematics, FORTH, Greece.} \\{\footnotesize
(In: {\it Journal of Elasticity (2015), 118-2, 141-185,} DOI: 10.1007/s10659-014-9488-z)}
}

\author{Alexey Miroshnikov\thanks
{Department of Mathematics and Statistics, University of
Massachusetts Amherst, USA. amiroshn@gmail.com } \; and \, Athanasios E. Tzavaras
\thanks{Division of Computer, Electrical, Mathematical Sciences \& Engineering,
King Abdullah University of Science and Technology (KAUST), Thuwal, Saudi Arabia.  athanasios.tzavaras@kaust.edu.sa}}

\date{}

\begin{document}

\maketitle

\begin{abstract}
\noindent

We consider the problem of dynamic cavity formation in isotropic compressible nonlinear elastic media. For the equations
of radial elasticity we construct self-similar weak solutions that describe a cavity emanating from a state of uniform deformation.
For dimensions  $d =2, 3$ we show that cavity formation is necessarily associated with a unique precursor shock.
We also study the bifurcation diagram and do a detailed analysis of the singular asymptotics associated to cavity initiation
as a function of the cavity speed of the self-similar profiles. We show that
for stress free cavities the critical stretching associated with dynamically cavitating solutions coincides with the critical stretching in the bifurcation diagram of equilibrium elasticity. Our analysis treats both stress-free cavities and cavities with contents.\\
\end{abstract}



%
%



\section{Introduction}

The motion of a continuous medium with nonlinear elastic response
is described by the system of partial differential equations
\begin{align}
&y_{tt} - \mathrm{div}  \frac{\del W}{\del F} (\nabla{y}) = 0 \label{MPDE}
\end{align}
where $y  : \RR^d \times \RR_{+} \to{\RR}^d$ stands for the motion, $F=\nabla y$ is the deformation gradient,
and we have employed the constitutive theory of  hyperelasticity,  $S = \frac{\del W}{\del F}(F)$,  that the
Piola-Kirchhoff stress $S$ is given as the gradient  of a stored energy function
\begin{equation*}
W \, : \, M^{d \times d}_{+}:=\{\,F\in \RR^{d\times d} : \det (F) >0 \,\} \, \longrightarrow \, \RR\,.
\end{equation*}

For isotropic elastic materials the stored energy reads  $W(F) = \Phi(v_1,v_2,\dots,v_d)$\,,
where $\Phi$ is a symmetric function of the eigenvalues $v_1,\dots,v_d$ of the positive square root $( F^{\T} F)^{\frac{1}{2}}$; see \cite{Antman,Tr}.
 In that case \eqref{MPDE} admits solutions that are radially symmetric motions,
\begin{equation*}
  y(x,t) = w(|x|,t)\frac{x}{|x|} \, , \quad R = |x| \, ,
\end{equation*}
and  are generated by solving for the amplitude $w:\RR_+ \times \RR_{+} \to \RR_+$
the scalar second-order equation
\begin{equation}\label{RPDEINTRO}
\begin{aligned}
       &w_{tt}
        =\frac{1}{R^{d-1}} \frac{\del}{\del R}\bigg(R^{d-1} \frac{\del\Phi}{\del v_1} \big ( w_R,\frac{w}{R},\dots,\frac{w}{R} \big )\bigg)
         -\frac{1}{R}(d-1)\frac{\del \Phi}{\del v_2} \big (w_R,\frac{w}{R},\dots,\frac{w}{R} \big )  \, .
\end{aligned}
\end{equation}
This equation admits  the special solution $w_h (R,t)=\lambda R$ corresponding to a homogeneous deformation of stretching $\lambda > 0$.
The question was posed \cite{Ball82} if discontinuous solutions of  \eqref{RPDEINTRO} can be constructed and it has
been tied to a possible explanation of the phenomenon of cavitation in stretched rubbers \cite{GL58, GT69}.

Ball \cite{Ball82} in a seminal paper proposed to use continuum mechanics for modeling cavitation
and used methods of the calculus of variations and bifurcation theory to construct cavitating solutions for the equilibrium
version of  \eqref{RPDEINTRO}:  There is a critical stretching $\lambda_{cr}$ such that for $\lambda < \lambda_{cr}$ the homogeneous deformation is
the only minimizer of the elastic stored energy; by contrast,  for $\lambda > \lambda_{cr}$ there exist nontrivial equilibria corresponding
to a (stress-free) cavity in the material with energy less than the energy of the homogenous deformation \cite{Ball82}.
We refer to \cite{MSP, SSP00, SSP08, LP09, MS11, MC14} and references therein for developments concerning
equilibrium or quasistatic cavitating solutions.

In an important development, K.A. Pericak-Spector and S. Spector \cite{Sp,Sp2}  use the self-similar
{\it ansatz}
\begin{equation}\label{SSIM}
  w(R,t) = t \SP \varphi(\tfrac{R}{t})
\end{equation}
to construct a weak solution for the dynamic problem \eqref{RPDEINTRO} that corresponds to  a
spherical cavity emerging at time $t = 0$ from a homogeneously deformed state. The cavitating solution
is constructed in dimension $d \geq 3$ for special classes of polyconvex energies \cite{Sp,Sp2} and sufficiently large initial stretching.
Remarkably,  the cavitating solution has lower mechanical energy than the associated homogeneously deformed state from where it emerges \cite{Sp},
and thus provides a striking example of nonuniqueness of entropy weak solutions (for polyconvex energies). {The dynamic cavitation problem is a little studied subject. Apart from \cite{Sp,Sp2}, there is an interesting almost explicit example of a dynamic
solution that oscillates constructed by Chou-Wang and Horgan \cite{CH89} for the dead load problem of an incompressible elastic material.
Due to the incompressibility constraint the response is markedly different from the compressible case: beyond a critical load
a cavity opens and then closes again, see \cite{CH89}.
The reader is referred  to Choksi \cite{CH97} for a discussion of the limit from compressible to incompressible response
in radial elasticity, and to Hilgers \cite{Hilgers12} for other examples of non-uniqueness in multi-dimensional hyperbolic conservation laws
due to radial point singularities.
}

The objective of the present work is to complement \cite{Sp,Sp2} by establishing various further properties of weak solutions describing dynamic
cavitation.
First, we prove that cavity formation is always associated with
a precursor shock, namely it is not possible to construct a cavitating solution that connects "smoothly" to a
uniformly deformed state.
Second, we study the bifurcation diagram for dynamically cavitating solution and provide a formula that determines
the critical stretch required for opening a cavity. The critical stretch turns out (for traction free cavities) to be the same as that predicted from
the equilibrium cavitation analysis of Ball \cite{Ball82}. In a companion paper \cite{GT14} we reassess the issue of nonuniqueness of weak solutions,
and show that local averaging of the cavitating weak solution contributes a surface energy when opening a cavity that renders the uniform deformation
the energetically preferred solution, see \cite{GT14} for details and comments on the ramifications.

We now provide an outline of the technical contents of the article:
Throughout we work with stored energies of the form
\begin{equation}\tag{H0}\label{spseis}
\begin{aligned}
&\Phi (v_1, v_2, ... , v_d) =  \sum_{i =1}^d g(v_i) + h (v_1 v_2
\dots v_d)
\end{aligned}
\end{equation}
where $ g(x)\in C^3[0,\infty), \;\; h(x)\in C^3(0,\infty)$ satisfy
\begin{align}
& g''(x)  > 0,  \quad  h''(x)  > 0   \, , \quad
 \lim_{x \to \infty} h(x) = +\infty
 \label{GHPROP1}\tag{H1}
 \\
 &g'''(x) \leq  0,  \quad h'''(x) < 0 \,.\label{GHPROP2} \tag{H2}
\end{align}
Hypothesis \eqref{GHPROP1} refers to polyconvexity \cite{Ball82}, while \eqref{GHPROP2} indicates elasticity with softening.
{
The reader is referred to Appendix \ref{app1} where properties of isotropic stored energies are reviewed.
}
\par\smallskip

Following \cite{Sp,Sp2}, we introduce the self-similar {\it ansatz} \eqref{SSIM} and the problem of cavity formation becomes to find
solutions of the problem
\begin{align}
\label{SSIMODE}
\bigg(s^2 - \frac{\del^2\Phi}{\del v_1^2 }\bigg)  \SP \ddot{\varphi}
&\;= \frac{d-1}{s}
\bigg [(\dot\varphi - \frac{\varphi}{s}) \frac{\del^2 \Phi}{\del v_1 \del v_2} + \frac{\del\Phi}{\del v_1} - \frac{\del \Phi}{\del v_2}\bigg]
\\
\label{initcav}
\varphi_0 &:=  \lim_{s\to 0_+}\varphi(s)   > 0
\end{align}
and to check whether such solutions can be connected to a uniformly deformed state, namely
\begin{align}
\label{SSIMTRIV}
\varphi(s) &= \lambda s \;\;\, \mbox{for} \;\; s> \sigma.
\end{align}
Here, $\varphi_0>0$ represents the speed of the cavity surface, $\lambda > 0$ the stretching of the (initially) uniform deformation and $\sigma$
is the shock speed.
We remark that  \eqref{SSIMODE} with \eqref{SSIMTRIV} admit the special solution $\bar{\varphi}=\lambda s$ corresponding
to a homogenous deformation ${y_h} (x) = \lambda x$; therefore, according to this scenario,   cavity formation
is associated  to nonuniqueness for the initial value problem of  the radial elasticity equation \eqref{RPDEINTRO}.

To make the problem \eqref{SSIMODE}-\eqref{initcav} determinate it is necessary to specify the value of the radial component of the Cauchy stress $\SP T_{rad}(0)\SP$ at the cavity surface. Two types of boundary conditions are pursued (see Section \ref{secCauchy}) corresponding to stress-free cavities
or to a cavity with content:
\begin{equation}
\label{CSTRESSINTRO}
\begin{aligned}
 \mbox{either} \quad  T_{rad}(0)&=0 && \Leftrightarrow \quad \lim_{s\to 0} \SP \dot{\varphi}\Big(\frac{\varphi}{s}\Big)^{d-1} \SP = \SP h'^{-1}(0):=H \quad \mbox{(stress-free cavity)}\\
 \mbox{or}  \quad  T_{rad}(0)&=G(\varphi_0) && \Leftrightarrow \quad \lim_{s\to 0} \SP \dot{\varphi}\Big(\frac{\varphi}{s}\Big)^{d-1} \SP = \SP h'^{-1}(G(\varphi_0)) \quad \mbox{(cavity with content)}\,.
\end{aligned}
\end{equation}
Under the growth condition $|g'(x)| \le C |x|^{d-2}$ (see  \eqref{GGROWTH} in section \ref{seccavi}) and for dimension $d \ge 2$,
the problem \eqref{SSIMODE}, \eqref{initcav} and \eqref{CSTRESSINTRO}
is desingularized at the origin and a solution $\varphi(s)$ is constructed (see Theorem \ref{CAVSOLEXISTSP}).
The question  arises whether this cavitating solution can be connected to the uniform deformation \eqref{SSIMTRIV} through a shock
(or through a sonic singularity). This leads to studying the algebraic equation
\begin{equation}
\label{rhintro}
\sigma = \left .  \sqrt{ \frac{\frac{\del \Phi}{\del v_1}( \dot \varphi, \frac{\varphi}{s}, ... , \frac{\varphi}{s} )-
    \frac{\del \Phi}{\del v_1}(\frac{\varphi}{s}, \frac{\varphi}{s}, ... , \frac{\varphi}{s})}{\dot \varphi -  \frac{\varphi}{s} } } \; \;  \right  |_{s = \sigma}
\end{equation}
which manifests the Rankine-Hugoniot jump condition. In Theorems \ref{EXISTCPT} and \ref{UNIQCPT} we show there
exists a unique $\sigma$ where the connection can be effected, and that the connection either happens  through a Lax shock or
through a sonic singularity ({\it i.e.} a point where the coefficient $(s^2 - \Phi_{11})$  in \eqref{SSIMODE} vanishes).
Then, in Theorem \ref{MAINRESULTD3}, we restrict to dimensions $d=2,3$ and exclude the possibility of a connection through a sonic singularity.
{
Our analysis is inspired and extends the results of \cite{Sp, Sp2} where  cavitating solutions are constructed for sufficiently large stretchings
$\lambda$.
In particular, we  show that, for dimensions $d=2,3$, it is impossible to connect a cavitating solution smoothly to a uniformly deformed state and thus
any cavitating solution is associated with a precursor shock.
}

The next objective is to study the bifurcation diagram of the  cavitating weak solution and determine
the critical stretching for dynamic cavitation. The bifurcation diagram is visualized as follows:  The boundary condition \eqref{CSTRESSINTRO}
is expressed for the specific volume $v(s) : = \det F = \dot \varphi \Big(\frac{\varphi(s)}{s}\Big)^{d-1}$ in  the general form $v(0)= V(\varphi_0)$. Given the cavity speed $\varphi_0 > 0$,
let $(\varphi,v)(s \SP; \varphi_0 , V(\varphi_0))$ be the  cavitating solution  (constructed in Section \ref{seccavi}) emanating from data $\varphi_0$,  $v_0 = V(\varphi_0)$.
 Denoting by  $\sigma = \sigma(\varphi_0 , V(\varphi_0))$ the connection point,  the associated stretching defines the map
\begin{equation}
\label{mapintro}
\varphi_0 \longmapsto \Lambda(\varphi_0 , V(\varphi_0) )   \quad \mbox{ where } \;  \;
\Lambda(\varphi_0 , V(\varphi_0) ) := \frac{\varphi \big ( \sigma \SP; \varphi_0 , V(\varphi_0) \big)}{\sigma } \, ,
\end{equation}
which is precisely the dynamic bifurcation diagram (see Fig. \ref{DYNSTAT} for a numerical computation of this map).
The limit $\lim_{\varphi_0 \to 0+} \Lambda(\varphi_0 , V(\varphi_0))$ will determine the critical stretching.
{The technique of recovering the bifurcation point $\lambda_{cr}$  by computing the cavitating solution and sending the inner radius of the cavity to zero
 is espoused in \cite{MS08}, where the authors use it to devise a numerical scheme for computing $\lambda_{cr}$ in equilibrium elasticity. }

To understand the limiting behavior of cavitating solutions as $\varphi_0 \to 0$, we introduce the rescaling
\begin{equation}
\label{rescalingintro}
\begin{aligned}
    \psi(\xi; \varphi_0 , V(\varphi_0) ) := \frac{\varphi \big(\varphi_0\xi\SP; \varphi_0 , V(\varphi_0) \big)}{\varphi_0},\quad
    \delta(\xi; \varphi_0 , V(\varphi_0) ) := v \big ( \varphi_0\xi\SP; \varphi_0 , V(\varphi_0) \big )\,
\end{aligned}
\end{equation}
which captures the inner asymptotics of the cavitating solution $(\varphi, v)$ to \eqref{SSIMODE}-\eqref{CSTRESSINTRO}.
Rescalings have been useful in the study of cavitation for equilibrium elasticity \cite{Ball82} and will play an instrumental
role in determining  the critical stretching for dynamic cavitation.

 It is proved
in Proposition \ref{proprescl}  that the rescaled solutions converge to a limiting profile,
\begin{equation*}
(\psi,\delta)(\xi\SP ; \varphi_0, V(\varphi_0))  \; \to \;  (\psi_0,\delta_0)(\xi\SP ; V(0)) \, , \quad \mbox{as} \; \;  \varphi_0  \to 0_+ \, ,
\end{equation*}
uniformly on compact subsets of $[0, \infty)$.

The limiting profile $(\psi_0 (\xi), \delta_0(\xi))$, where $\delta_0 = \psi'_0 \big ( \frac{\psi_0}{\xi} \big )^{d-1}$, is defined on $[0,\infty)$ and solves
the initial value problem
\begin{equation}\label{LIMODEINTRO}
\begin{aligned}
-  \frac{ \del^2 \Phi }{\del v_1^2}   ({\psi_0'},   \tfrac{\psi_0}{\xi}, ... ,  \tfrac{\psi_0}{\xi})  \, \psi_0'' &=
\frac{d-1}{\xi}\Big({\psi_0'} - \frac{\psi_0}{\xi}\Big)
\bigg [ \frac{\del^2 \Phi}{\del v_1 \del v_2}
+   \frac{ \frac{\del\Phi}{\del v_1} - \frac{\del \Phi}{\del v_2} }{ {\psi_0'} - \tfrac{\psi_0}{\xi} } \bigg] ({\psi_0'},   \tfrac{\psi_0}{\xi}, ... ,  \tfrac{\psi_0}{\xi}) \, ,
\\
\psi_0 (0) &= 1 \, ,
\\
\delta_0 (0) &= V(0) \, .
\end{aligned}
\end{equation}
The solvability of \eqref{LIMODEINTRO} and properties of its solutions are discussed in Proposition \ref{LIMSOLLMM}, where it is in particular
shown that the (inner) solution is associated with a critical stretching at infinity
\begin{equation}
\label{infbeh}
\Lambda_0 (V(0)) := \lim_{\xi \to \infty} \frac{\psi_0 (\xi ; V(0) )}{\xi} \, .
\end{equation}
Equation \eqref{LIMODEINTRO}$_1$ is precisely the equation  describing cavitating solutions in equilibrium radial elasticity,
suggesting  that the critical stretch for dynamic cavitation and equilibrium cavitation might conceivably coincide. The critical stretch $\lambda_{cr}$
for cavitation in equilibrium radial elasticity is studied in \cite[Section 7.5]{Ball82} where various representation
formulas for $\lambda_{cr}$ are established. In section \ref{STATCONNECT} we pursue this analogy, we show that for a stress-free
 cavity $ \Lambda_0(H) = \lambda_{cr}$, and establish representation formulas for the
critical stretch and corresponding lower bounds.

Finally, in Theorem \ref{BIFPNTTHM}, we study  the behavior of the cavitating solution $\varphi (\cdot ; \varphi_0 , V(\varphi_0) )$ and the associated stretch
\eqref{mapintro}  as the cavity speed $\varphi_0 \to 0$. We establish that
\begin{equation*}
  \lim_{\varphi_0 \to 0_+} \Lambda(\varphi_0 , V(\varphi_0)) = \Lambda_0(V(0)) \, ,
\end{equation*}
where $\Lambda_0(V(0))$ is given by \eqref{infbeh},  that the speed and the strength of the precursor shock satisfy
\begin{equation*}
\begin{aligned}
&\lim_{\varphi_0\to 0_+} \sigma(\varphi_0 , V(\varphi_0) ) = \sqrt{\frac{\del^2\Phi}{\del v_1^2}\big(\Lambda_0(V(0)),\dots,\Lambda_0(V(0))\big)}
\\
&\lim_{\varphi_0\to 0_+} \Big[ \frac{\varphi}{s}-\dot{\varphi} \Big] (\sigma(\varphi_0 , V(\varphi_0)))=0\, ,
\end{aligned}
\end{equation*}
and that
$$
\varphi \big (s ; \varphi_0 , V(\varphi_0) \big ) \to  \Lambda_0(V(0)) s \, , \quad  \mbox{as  $\varphi_ 0 \to 0_+$}.
$$
Our analysis proves that the critical stretching for equilibrium and dynamic cavitation coincide.

The structure of the article is as follows: In Section \ref{sec2} we introduce the equations of radial elasticity for isotropic elastic materials.
In Section \ref{secselfsimilar} we derive the equations for self-similar solutions of radial elasticity, describe various special solutions, and
present the problem of cavitation. The analysis of Section \ref{secselfsimilar} follows the ideas and extends the analysis
 of \cite{Sp,Sp2} to  a more general class of (polyconvex) stored energies
and to boundary conditions of cavities with content. Section \ref{necshock}  and Section \ref{secbifurc} contain the main new results.
In Section  \ref{necshock} we establish various  properties of weak solutions describing cavity formation from a homogeneously deformed state. In Section \ref{secbifurc} we study the bifurcation curves associated with cavitating weak solutions and establish the properties of the critical stretching and its relation to the critical stretching predicted by the equilibrium elasticity equation. The Appendix lists
some properties of radial deformations, and collects information on stored energies that is widely used in various places of the text.

\section{The equations of radial elasticity}
\label{sec2}

The stored energy of an isotropic elastic material has to satisfy the symmetry requirements
\begin{align*}
\mbox{frame indifference} \quad   &W(QF)  = W(F),  \quad \forall \;  Q \in SO(d)
\\
\mbox{isotropy}  \quad  &W(F)  = W(FQ),     \quad  \forall \; Q \in SO(d).
\end{align*}
where $Q$ is any proper rotation.
These requirements are equivalent to
\begin{equation*}
 W(F)=\Phi (v_1,v_2,\dots, v_d)
\end{equation*}
where $\Phi(v_1,v_2,\dots, v_d) : \RR^d_{++} \to \RR$
is a symmetric function of its arguments and $v_1,\dots,v_d$ are the eigenvalues of
$(F^{\T}F)^{\frac{1}{2}}$ called principal stretches \cite{Antman,Ball82,Tr}.

For isotropic materials the system of elasticity \eqref{MPDE} admits radial solutions of the form
\begin{equation}
\label{radial}
y(x,t) = w( R,t) \frac{x}{R} \quad \mbox{with $R = |x|$}.
\end{equation}
The deformation gradient is computed by
\begin{equation}
\label{RADGRADY}
    \nabla y =  w_R \frac{x \otimes x}{R^2}  +
      \frac{w}{R}  \left (  {\Id} - \frac{x \otimes x}{R^2} \right )
\end{equation}
and has principal stretches $v_1=w_R,\;\; v_2 =  \dots = v_d = \frac{w}{R}$. Using results
on spectral representations of functions of matrices one computes
the first Piola-Kirchhoff stress \cite[p.564]{Ball82},
\begin{equation}\label{PKSTRESSRAD}
\begin{split}
        S(\nabla {y})&=
         \Phi_{1} \big ( w_R, \frac{w}{R} , ... , \frac{w}{R} \big )   \;  \frac{x \otimes x}{R^2} + \Phi_2 \big ( w_R, \frac{w}{R} , ... , \frac{w}{R} \big )
          \;  \Bigl( I - \frac{x \otimes x}{R^2} \Bigr) \, ,
\end{split}
\end{equation}
where we used the notation $ \Phi_{1}  \equiv \frac{\del \Phi}{\del v_1} $,
$ \Phi_{2}  \equiv \frac{\del \Phi}{\del v_2} $ and the symmetry property \eqref{phiprop1}.
(The reader is referred to Appendix \ref{app1} for properties of the stored energies $\Phi$ and details on the notation used throughout).
Using the above formulas one computes that the amplitude $w$  of the radial motion \eqref{radial} is generated by solving
 the second-order partial differential equation
\begin{equation}
\label{radialelas}
\begin{aligned}
       w_{tt}
        &=\frac{1}{R^{d-1}} \frac{\del }{\del R}\bigg(R^{d-1} \frac{\del \Phi}{\del v_1} \big ( w_R,\frac{w}{R}, ..., \frac{w}{R} \big )\bigg)
         -\frac{d-1}{R}  \frac{\del \Phi}{\del v_2} \big (w_R,\frac{w}{R}, ..., \frac{w}{R} \big ).
\end{aligned}
\end{equation}
 In order for solutions to be
interpreted as elastic motions one needs to impose the requirement
$$
\det F = w_R \big (\frac{w}{R} \big )^{d-1} > 0
$$
on solutions of \eqref{radialelas}, which for radial motions suffices to
exclude interpenetration of matter.

Equation \eqref{radialelas} can also be derived by considering the action functional for
radial, isotropic elastic materials, defined as the difference between
kinetic and potential energy
\begin{equation*}
I[w] := \int_0^T \int_0^1 R^{d-1} \Bigl( \tfrac{1}{2} w_t^2  - \Phi(w_R , \frac{w}{R},
..., \frac{w}{R} ) \Bigr)  dR \SP dt \SP .
\end{equation*}
Critical points of the functional $I[w]$ are obtained by computing the first variation
and setting it to zero,
$$
\frac{d}{d \delta} \Big |_{\delta = 0} I [ w + \delta \psi] = 0,
$$
which gives the weak form of \eqref{radialelas},
\begin{equation*}
\int_0^T \int_0^1 R^{d-1} \Bigl( w_t \psi_t - \Phi_1 \big ( w_R,\frac{w}{R}, ..., \frac{w}{R} \big ) \SP \psi_R -  \frac{d-1}{R} \SP
\Phi_2 \big ( w_R,\frac{w}{R}, ..., \frac{w}{R} \big ) \SP \psi \Bigr) dR \SP    dt \SP = 0 \, .
\end{equation*}


Finally, \eqref{radialelas} can be expressed as
a first order system  by introducing the variables
\begin{equation}\label{1ORDERVAR}
a = w_R \, , \quad b = \frac{w}{R} \, ,  \quad v = w_t
\end{equation}
where $u$ is the (longitudinal) strain, $b$ is the transverse strain,  and $v$ is the velocity in the radial direction.
It is expressed as the equivalent first order system
\begin{equation}
\label{elassys}
\begin{aligned}
a_t  &= v_R
\\
v_t  &=   \del_R \Big(  \Phi_1 (a, b, ...,  b) \Big)
   + \frac{d-1}{R} \SP  \Big  ( \Phi_1 (a,b, ... , b) - \Phi_2 ( a, b, ... , b ) \Big )
\\
b_t &= \frac{v}{R} .
\end{aligned}
\end{equation}
subject to the involution $(b R)_R =a$.  This is a system of balance laws with geometric singularity at $R=0$.
Under the hypothesis $\Phi_{11} > 0$ the system \eqref{elassys} is hyperbolic (see  \cite[Def 3.1.1]{daf}  for the usual definition).
The characteristic speeds
$\lambda_{\pm} = \pm \sqrt{ \Phi_{11} }$ are genuinely nonlinear, while $\lambda_0 = 0$
is linearly degenerate, see \cite{LAX, daf}.
The eigenvalues and the corresponding right and left
eigenvectors of the flux of the system \eqref{elassys} are given by
\begin{equation*}
\begin{aligned}
\lambda_{+} &=  \sqrt{\Phi_{11}} \,, \;\;
 r_{+} = \big(1, \; \sqrt{\Phi_{11}}, \; 0
\big)^{\T} \,, \;\;    l_{+}= \big(\sqrt{\Phi_{11}},\; 1,\;
{(d-1)\Phi_{12}}/{\sqrt{\Phi_{11}}}
\big)\\
\lambda_{-} &= - \sqrt{\Phi_{11}}  \,,
\quad r_{-} = \big( 1, \; -\sqrt{\Phi_{11}}, \;
0\big)^{\T} \,, \;\;
  l_{-}= \big(\sqrt{\Phi_{11}}, \; -1, \;
{(d-1)\Phi_{12}}/{\sqrt{\Phi_{11}}}
\big)\\
\lambda_{0} &= 0,\;\;
r_0 = \big( (d-1)\Phi_{12} \,, \;  0, \;  -\Phi_{11} \big)^{\T} \,,
\quad  l_0 =\big(0, \;  0, \; 1 \big).
\end{aligned}
\end{equation*}

\section{ The cavitating solution of Pericak-Spector and Spector}
\label{secselfsimilar}

We are interested in  \eqref{radialelas} subject to the initial-boundary conditions
\begin{equation}
\label{icbc}
\left\{\begin{aligned}
w(R,0) &= \lambda R \\
w(R,t) &= \lambda R \;,\; \mbox{for} \;\; |R| >\bar{\sigma}t\,.
\end{aligned}\right.
\end{equation}
The symmetry of $\Phi$ implies that $\Phi_1 ( \lambda ,\lambda) = \Phi_2 ( \lambda , \lambda)$ and the
 homogeneous deformation  $w_h (R,t)=\lambda R$ is a special equilibrium solution of \eqref{radialelas} associated to the stretching $\lambda > 0$.
To obtain additional solutions, it was suggested in \cite{Sp} to exploit the invariance of
\eqref{radialelas}, \eqref{icbc}  under the family of the scaling transformations
$
w_\lambda (R,t) = \lambda w ( \lambda R , \lambda t)
$
and to seek solutions in self-similar form
\begin{equation}
\label{ssansatz} w(R,t) = t \SP \varphi \Big( \frac{R}{t} \Big).
\end{equation}
Introducing the ansatz \eqref{ssansatz} to \eqref{radialelas}, and using the notations $s=\frac{R}{t}$ and $\dot{} = \frac{d}{ds}$,
 it turns out that  $\varphi(s)$ satisfies
the singular second-order ordinary differential equation
\begin{equation}
\label{ssode}
(s^2 - \Phi_{11} )  \ddot{\varphi} = \frac{d-1}{s} (\dot\varphi - \frac{\varphi}{s} )
\Big [ \Phi_{12} + \frac{\Phi_1 - \Phi_2}{\dot\varphi - \frac{\varphi}{s}} \Big ]  \, .
\end{equation}
Henceforth, we will be using the short hand notations
$$
\Phi_{i}  \big ( \dot \varphi, \frac{\varphi}{s} \big )  \equiv \frac{\del \Phi}{\del v_i}   \big ( \dot \varphi, \frac{\varphi}{s} , ... , \frac{\varphi}{s} \big )  \, \quad
\Phi_{i j }  \big ( \dot \varphi, \frac{\varphi}{s} \big )  \equiv \frac{\del^2 \Phi}{\del v_i \del v_j}   \big ( \dot \varphi, \frac{\varphi}{s} , ... , \frac{\varphi}{s} \big )
$$
and so on for higher derivatives. We refer to Appendix \ref{app1} for details, and caution the reader that the notation together with
the symmetry properties \eqref{phiprop1}, \eqref{phiprop2}
has implications on the differentiation of such formulas.

Moreover, we introduce the variables  $a,b$ defined in analogy to \eqref{1ORDERVAR} by
\begin{equation*}
\label{trans} a = \dot\varphi \, , \quad b = \frac{\varphi}{s}
\end{equation*}
and rewrite \eqref{ssode} in the form of the first order system
\begin{equation}
\label{ssys} \left\{
\begin{aligned}
\big ( s^2 - \Phi_{11} (a,b) \big ) \dot{a} & \, = \, \frac{(d-1)}{s}  (a - b)
P(a,b)
\\
\dot{b} & \, = \, \frac{1}{s} (a-b)
\end{aligned} \right.
\end{equation}
where
\begin{equation}\label{PDEF}
P(a,b) =
\begin{cases}
\Phi_{12}( a,b) + \frac{\Phi_1(a,b) - \Phi_2(a,b)}{a-b}\,, & a<b\\
\Phi_{11}(b,b)\,, & a=b
\end{cases}
\end{equation}
is a continuous function on $\big\{(a,b)\in \RR^2: 0 < a \leq b \big\}$.

\medskip

In analogy to the standard theory \cite{LAX} of the Riemann problem for conservation laws, it is instructive to
classify elementary solutions of \eqref{ssys}. There are three classes of special solutions:

\medskip

{\bf  (a) Uniformly deformed states.}
A special class of solutions of \eqref{ssys} are the constant states  $a = b = constant$, which yield a uniform
deformation $w_h (R)  = \lambda R$ for the original system.

\medskip

{\bf (b) Continuous solutions.} The balance of the convective and the production terms in \eqref{ssys}
leads to a class of solutions that are continuous  (which are not present in homogeneous conservation laws
and are of different origin than the rarefaction waves).
These will be the main object of study here. There are two features of \eqref{ssys} that need
to be addressed by the analysis: (i) the geometric singularity at $s=0$, and (ii) the difficulty
emerging from a potential free boundary at the sonic curve $s = \pm \sqrt{\Phi_{11}(a,b)}$.
It is well known that the resolution of the  Riemann problem for
multi-dimensional hyperbolic systems leads to  systems that change type across sonic-curves
 in the self-similar variables. The analog of this phenomenon for radial solutions
leads to singular ordinary differential equations across the sonic lines.

\medskip

{\bf (c) Shocks.}
One may express the system \eqref{ssys} in the equivalent form
\begin{equation}
\label{ssys2}
\begin{aligned}
\frac{d}{ds} \Big ( s^2 \SP a - \Phi_1 (a, b) \Big ) & \SP = \SP 2 s \SP a +
\frac{d-1}{s} \big( \Phi_1 (a,b) - \Phi_2 (a,b) \big)
\\
\frac{db}{ds} & \SP = \SP \frac{1}{s} (a-b).
\end{aligned}
\end{equation}
Two smooth branches of solutions to \eqref{ssys2} might be connected through
a jump discontinuity at $s = \sigma$ provided that the Rankine-Hugoniot jump conditions
\begin{align}\label{SSJCOND}
b_- = b_+ = : b\,,  \quad \sigma^2 = \frac{ \Phi_1 (a_+ ,  b) - \Phi_2 (a_- ,
b) }{ a_+ - a_-}\,
\end{align}
are satisfied, where
\begin{equation*}
(a_-,b_-) = \lim_{s\to \sigma_-} (a,b)(s)\,, \quad (a_+,b_+) = \lim_{s\to \sigma_+} (a,b)(s).
\end{equation*}

According to the Lax shock admissibility criterion (see \cite{LAX} where the criterion was introduced or \cite[Secs 8.3, 9.4]{daf}),
a shock of the $2^{nd}$ characteristic family will be admissible  if
\begin{equation}\label{LX2FAM}
\begin{aligned}
\sqrt{\Phi_{11}(a_+,b)} \, <  \, \sigma_{+} = \sqrt{\frac{ \Phi_1 (a_+ , b) - \Phi_2 (a_-
, b) }{ a_+ - a_-}} \, < \, \sqrt{\Phi_{11}(a_-,b)} \,.
\end{aligned}
\end{equation}
In particular,
\begin{align*}
&\mbox{if \, $\Phi_{111}>0$, \, then \eqref{LX2FAM} is equivalent to  $a_+ < a_-$} \notag\\
&\mbox{if \, $\Phi_{111}<0$, \, then \eqref{LX2FAM} is equivalent to  $a_+ > a_-$} \, .\label{LX2FAM2}
\end{align*}
Similarly, shocks of the $1^{st}$ characteristic family are  admissible via the Lax criterion if
\begin{equation}\label{LX1FAM}
\begin{aligned}
-\sqrt{\Phi_{11}(a_+,b)} \, <  \, \sigma_{-} = -\sqrt{\frac{ \Phi_1 (a_+ , b) - \Phi_2
(a_- , b) }{ a_+ - a_-}} \, < \, -\sqrt{\Phi_{11}(a_-,b)} \,.
\end{aligned}
\end{equation}
in which case
\begin{align*}
&\mbox{if \, $\Phi_{111}>0$, \, then \eqref{LX1FAM} is equivalent to $a_- < a_+$}\\
&\mbox{if \, $\Phi_{111}<0$, \, then \eqref{LX1FAM} is equivalent to $a_- > a_+$} \, .
\end{align*}
For radial motions shocks of the $2^{nd}$ characteristic family are outgoing while shocks of the
$1^{st}$ characteristic family are incoming to the origin. For the cavitation problem, it is natural to restrict to outgoing shocks
and the kinematics of the cavity dictates that  $a_- < a_+$. Therefore, we impose the
condition $\Phi_{111}<0$ which corresponds to softening elastic response.
Softening refers to the property that the elastic modulus decreases with an increase
of the longitudinal strain and plays an important role in cavitation analysis.

\medskip

\subsection{The cavitating solution}
\label{seccavi}
We next consider the problem of cavitation and discuss the continuous type of solutions in this context.
We employ a  constitutive relation of polyconvex class
\begin{equation*}\tag{H0}
\begin{aligned}
&\Phi (v_1, v_2, ... , v_d) =  \sum_{i =1}^d g(v_i) + h (v_1 v_2  \dots v_d) \, ,
\end{aligned}
\end{equation*}
where $g\in C^3[0,\infty)$, $h  \in C^3(0,\infty)$ satisfy \eqref{GHPROP1} and \eqref{GHPROP2}
and thus  $\Phi_{11} > 0$ and  $\Phi_{111}< 0$. Hypothesis \eqref{GHPROP1}
alludes to polyconvexity of the stored energy while Hypothesis \eqref{GHPROP2} manifests softening elastic response.

A stored energy of the form \eqref{spseis}  with $g(x) = \frac{1}{2} x^2$ was used in \cite{Sp} to establish cavitation for  $d\ge 3$.
The  generalization presented in \eqref{spseis} is necessary in order to handle the case of $d=2$,
as the hypothesis of quadratic growth is too strong to allow for a cavity when $d=2$.
The ideas presented in this section closely follow the discussion of \cite{Sp,Sp2}, nevertheless they are presented here
first for the reader's convenience but also  to set up the landscape for the forthcoming analysis in the following sections.

The differential equation \eqref{ssode} is expressed as
\begin{equation}\label{ODESP}
Q(\dot{\varphi},\tfrac{\varphi}{s},s) \SP \ddot{\varphi}
=\frac{(d-1)}{s}\Bigl(\dot{\varphi} -
\frac{\varphi}{s}\Bigr)P(\dot{\varphi},\tfrac{\varphi}{s})
\end{equation}
or equivalently
\begin{equation}
\label{SSYSSP} \left\{
\begin{aligned}
Q(a,b,s)\SP \dot{a} & \SP = \frac{(d-1)}{s} (a-b)P(a, b)
\\
\dot{b} & \SP = \SP \frac{1}{s} (a-b)
\end{aligned}\right.
\end{equation}
where
\begin{align}
Q(a,b,s)  &  \,\, =  s^2 - \Phi_{11}(a,b) \,\stackrel{ \eqref{spseis} }{=}
\, s^2 - \big[ g''(a) + b^{2d-2}h''\big(ab^{d-1})\big] \label{QDEFSP}.
\end{align}
%

\noindent{\bf Desingularization at the origin.} We next transform  \eqref{ODESP} into a system for the quantities
\begin{equation}\label{VDEF}
    \varphi(s), \,\, v(s) = \dot{\varphi}\Bigl(\frac{\varphi}{s}\Bigr)^{d-1} \;\;\;
    \mbox{with data} \;\;\; \varphi(0)=\varphi_0 > 0, \,\,\, v(0)=v_0 > 0 \, ,
\end{equation}
henceforth restricting to stored energies of class \eqref{spseis}. A lengthy but straightforward calculation shows that
$(\varphi , v)$ satisfies the initial-value problem
\begin{equation}\label{ODEIVPSYS}
\left\{
\begin{aligned}
\dot{\varphi} &= v\Big(\frac{s}{\varphi}\Big)^{d-1}\\
\dot{v} & = \bigg(\frac{d-1}{\varphi}\bigg) \frac{ \Big( \frac{s}{\varphi}\Big)^{2d-3}v
\Big( v \big(\frac{s}{\varphi}\big)^{d} - 1 \Bigr)\Big[s^2 - g''( v
(\tfrac{s}{\varphi})^{d-1})
\Big]}
{\Bigl \{ -h''(v)+\Big [ s^2-g''(v(\tfrac{s}{\varphi})^{d-1})\Big ] \big(\frac{s}{\varphi}\big)^{2d-2}\Bigr \}
\SP }
\\[2pt]
&\qquad +\bigg(\frac{d-1}{\varphi}\bigg)\frac{ \big( \frac{s}{\varphi}\big)^{d-2}
\Big[g'( v (\tfrac{s}{\varphi})^{d-1})-g'(\frac{\varphi}{s}) \Big]}
{\Bigl \{ -h''(v)+\Big [ s^2-g''(v(\tfrac{s}{\varphi})^{d-1})\Big ] \big(\frac{s}{\varphi}\big)^{2d-2}\Bigr \}
\SP }
\\
\varphi(0)&=\varphi_0>0\\
v(0)&=v_0>0.
\end{aligned}\right.
\end{equation}

\par\smallskip

In view of \eqref{GHPROP1} and the assumption $d \geq 2$,  the
only term on the right-hand side of \eqref{ODEIVPSYS} that might be singular  at $s=0$ is the term
$g'(\tfrac{\varphi}{s})(\frac{s}{\varphi})^{d-2}$. This motivates to impose the growth condition
\begin{equation}\label{GGROWTH}\tag{H3}
\lim_{x\to \infty} \bigg(\frac{g'(x)}{x^{d-2}}\bigg) = \gamma \geq 0 \,.
\end{equation}
Doing that the emerging system is not singular and one may apply the standard existence theory for ordinary
differential equations to obtain

\begin{lemma}\label{IVPEXISTLMM}
Suppose  $d \geq 2$ and hypotheses \eqref{spseis}, \eqref{GHPROP1}, and
\eqref{GGROWTH} hold. Then there exists a unique solution of the system \eqref{ODEIVPSYS}
defined on a maximal interval of existence.
\end{lemma}

{
\begin{proof}
Under hypotheses  \eqref{spseis}-\eqref{GHPROP1} we have $g''(0) > 0$, $h''(v_0)>0$. Moreover,
by \eqref{GGROWTH}, the limit $\lim_{x\to \infty} (g'(x)x^{2-d})$ exists and is finite and the right hand side of \eqref{ODEIVPSYS}
is continuous for $s \in [0,1]$ (up to the boundary $s=0$). A careful review of the various terms indicates that the initial value problem
\eqref{ODEIVPSYS} is expressed as
\begin{equation}\label{ODEIVPSYSN}
\left\{
\begin{aligned}
\dot{\varphi} &=  A(s, v, \varphi)
\\
\dot{v} & = B(s, v, \varphi)  +  C(s, v, \varphi)  H( s, \varphi) \\
\varphi(0)&=\varphi_0>0\\
v(0)&=v_0>0.
\end{aligned}\right.
\end{equation}
where  $A, B, C :  [0,1] \times \mathcal N ( v_0, \varphi_0 ) \to \RR$ are $C^1$ functions on $[0,1] \times \mathcal N( v_0, \varphi_0 ) $
with $\mathcal N$ a neighbourhood of $( v_0, \varphi_0 )$. The term
\begin{equation}
H(s,  \varphi) =  \begin{cases}
 \big ( \frac{s}{\varphi} \big )^{d-2} g'  \big ( \frac{\varphi}{s} \big ) & s > 0
 \\
 \gamma & s=0
 \end{cases}
\end{equation}
carries the singular behaviour in $s$ and by \eqref{GGROWTH} it is continuous on $[0,1] \times \mathcal N( v_0, \varphi_0 ) $.

Moreover, the assumptions  $g''>0$, $g''' \leq 0$ imply
\begin{equation*}
0 \leq  g''(x)x \leq \int_0^1 g'(sx)x ds = g'(x)-g'(0)  \, , \quad  x > 0
\end{equation*}
and,  using once again \eqref{GGROWTH}, it automatically implies
\begin{equation}\label{GGRCONS}
 0 \leq \limsup_{x \to \infty}  \bigg(\frac{g''(x)}{x^{d-3}} \bigg) \leq \gamma_0 = \left\{
 \begin{aligned}
 \gamma-g'(0)\SP, \quad & d =2 \\
 \gamma,\quad & d\geq 3\,.
 \end{aligned}\right.\,
\end{equation}
Now observe that
$$
\Big | \frac{ \partial H}{\partial \varphi} (s, \varphi)  \Big | =  \frac{1}{\varphi} \Big | g'' \big ( \tfrac{\varphi}{s} \big )  \big( \tfrac{s}{\varphi} \big )^{d-3}
       - (d-2)  g' \big ( \tfrac{\varphi}{s} \big )  \big( \tfrac{s}{\varphi} \big )^{d-2} \Big | \le C
$$
for $s \in [0,1]$ and $\varphi \in \mathcal{ N} (\varphi_0)$ a suitable neighbourhood of $\varphi_0 > 0$.
The standard existence and uniqueness theory for systems of ordinary differential equations then provides the result.
\end{proof}
}

\begin{remark}\rm
The sign requirements $g''>0, \, g'''\leq 0$ in hypotheses \eqref{GHPROP1}, \eqref{GHPROP2} place the restriction that $\limsup_{x\to \infty} |\frac{g'(x)}{x}|< \infty$.
Also,  if  $d=2$ then \eqref{GGROWTH} enforces that $\lim_{x\to\infty}|g'(x)| = \gamma<\infty$. This is consistent with \eqref{GHPROP1}-\eqref{GHPROP2}.
\end{remark}

\noindent{\bf Boundary data at the cavity surface.}\label{secCauchy} A natural assumption motivated from mechanical considerations is to impose that the radial Cauchy stress vanishes
at the cavity surface. Using the standard formula relating
the Cauchy stress tensor $T$ to the Piola-Kirchhoff stress $S$  ({\it e.g.} \cite{Ball82,Antman})
\begin{equation*}
    T(F) = \frac{1}{\det F} S(F) F^{\T}
\end{equation*}
and \eqref{PKSTRESSRAD} it follows that for  $\nabla y$ given by \eqref{RADGRADY} we have
\begin{equation*}
\begin{aligned}
T(\nabla y) = \frac{1}{w_R \big( \frac{w}{R}\big)^{d-1}} \bigg[\Phi_1 \SP w_R
\frac{x\otimes x}{R^2}+ \Phi_2 \frac{w}{R} \Big( \Id - \frac{x \otimes x}{R^2}\Big)
\bigg]\,.
\end{aligned}
\end{equation*}
The radial component of the Cauchy stress is given by
\begin{equation*}
T_{rad}(s)  :=   \frac{x}{R} \cdot  T(\nabla y) \frac{x}{R} \stackrel{ \eqref{spseis} }{=} \Big(\frac{s}{\varphi}\Big)^{d-1} \bigg[ g'(\dot{\varphi}(s)) +
\Big(\frac{\varphi}{s}\Big)^{d-1}h'(v(s))\bigg] \,.
\end{equation*}
For the solution $\varphi, v$  of \eqref{ODEIVPSYS} it is easy to see that
\begin{equation*}
\frac{\varphi}{s} \sim \frac{\varphi_0}{s}, \quad  \dot{\varphi}(s) \sim
    v_0\Big(\frac{s}{\varphi_0}\Big)^{d-1} \, \quad \, \mbox{as} \, \quad s\to 0_+
\end{equation*}
and therefore
\begin{equation}\label{TATZERO}
    \lim_{s\to 0_+}T_{rad}(s) =  h'(v_0) \,.
\end{equation}
This motivates to impose the growth condition
\begin{align}
h'(x) \to -\infty \; \;  \mbox{as} \quad x \to 0_+ \, ,
\qquad
h'(x) &\to +\infty \; \;  \mbox{as} \quad x \to +\infty  \, .
\qquad
\label{DHATINF}\tag{H4}
\end{align}
Under \eqref{GHPROP1}, \eqref{DHATINF} the inverse $h'^{-1}$ is a
well-defined function on $\RR$ and  the boundary condition becomes
\begin{equation}\label{CSTRESSBND1}
    T_{rad}(0)=0 \,\,\,\mbox{is equivalent to} \,\,\, v_0=H:=h'^{-1}(0)\,.
\end{equation}

One may consider more general boundary conditions that are referred in \cite{Sp} as cavities with content
and require that $T_{rad}(0)=G(\varphi_0)$, where $G$ is some prescribed function.
Such conditions could model at a phenomenological level the effect of remnant plasticity inside the cavity, and are postulated
in analogy to the form of kinetic relations in the motion of phase boundaries. For physical reasons the remnant plasticity at the cavity
should correspond to tensile forces, which dictates that $G(\varphi_0) > 0$.  One checks that
\begin{equation}\label{CSTRESSBND2}
    T_{rad}(0)=G(\varphi_0) \,\,\, \mbox{is equivalent to} \,\,\, v_0=h'^{-1}(G(\varphi_0)).
\end{equation}
It is not entirely clear if such an assumption is mechanically justified, nevertheless it can be analyzed by the mathematical
theory at no additional effort. Note that both \eqref{CSTRESSBND1} and \eqref{CSTRESSBND2} decrease the freedom of
the data by one degree. For the bifurcation analysis in section \ref{secbifurc} we assume that $G(\varphi_0)$ is continuous at $\varphi_0 = 0$.
This implies that  $V(x) := h'^{-1}(G(x))$ is also continuous at $x=0$.

\par\smallskip

\noindent{\bf A class of $C^2$ self-similar solutions.} We now construct a class of $C^2$ self-similar solutions to \eqref{ODESP}.
Proceeding along the lines of  \cite[Thm 5.1]{Sp} we have:

\begin{theorem}
\label{CAVSOLEXISTSP}
Assume that $d\geq 2$,  $\Phi$ satisfies \eqref{spseis}-\eqref{GGROWTH} and let $\varphi_0>0$, $v_0>0$. Then, there exists  a unique solution $\varphi$ of \eqref{ODESP} satisfying the initial data \eqref{VDEF} and defined on a maximal interval of existence $[0,T)$, with $T<\infty$. The solution has the following properties:
\begin{itemize}
\item [$(i)$] \; $(\varphi, v)$ solves \eqref{ODEIVPSYS} and there holds

\begin{equation}\label{PHIATZERO}
\begin{aligned}
\bvarp{s}  \, \sim \, \frac{\varphi_0}{s} \quad \mbox{and} \quad
\avar (s)  \, \sim \, v_0\Big(\frac{s}{\varphi_0}\Big)^{d-1} \, \quad \, \mbox{as} \quad
s\to 0_+ .
\end{aligned}
\end{equation}


\item [$(ii)$] \; $\avar$, $\bvar$, $(\avar -\bvar)$ are strictly monotonic and satisfy
\begin{align}\label{MONT1}
&\ddot{\varphi}(s)>0, \quad \frac{d}{ds}\Big(\bvar\Big) < 0, \quad \frac{d}{ds}\Big(\avar
- \bvar\Big) > 0  \quad \mbox{on} \quad 0<s<T\,,\\[3pt]
\label{MONT2}
& 0 <  \dot{\varphi}(s) < \dot{\varphi}(t) < \frac{\varphi(t)}{t} < \frac{\varphi(s)}{s} \qquad  \mbox{for} \qquad  \SP 0<s<t<T\,,
\end{align}
and
\begin{equation} \label{QSIGN}
\qquad Q(\avar,\bvart,s) < 0, \quad  \avar (s) - \bvarp{s} <0 \, \qquad \mbox{for} \qquad 0<s<T .
\end{equation}

\item[$(iii)$]
The following limits exist
\begin{equation}\label{QABLIM}
\lim_{s\to T_-}Q(\avar,\bvart,s)=0, \quad
\lim_{s\to T_-}\Big(\avar - \bvar \Big) = c_0 \leq 0\\
\end{equation}
(but it is not known if $c_0<0$).
\end{itemize}

\end{theorem}

\begin{proof}
By Lemma \ref{IVPEXISTLMM}   there exists a unique local solution $(\varphi,v)$ of  \eqref{ODEIVPSYS}
 with $\varphi(0)=\varphi_0>0$, $v(0)=v_0>0$
which of course satisfies (i). Next, we observe that $(a,b)=(\dot{\varphi},\frac{\varphi}{s})$ solves the system \eqref{SSYSSP} where, in view of
\eqref{QDEFSP} and \eqref{PHIATZERO},  we have $Q(0) <0$. This together with \eqref{PHIATZERO}
implies that $Q(s)<0$ and $0<a(s)<b(s)$ for $s\in (0,\eps]$ for sufficiently small $\eps>0$.
We now check that $\Phi$ satisfying \eqref{spseis}-\eqref{GGROWTH}
has the properties
\begin{equation}
\label{DPHISIGN1}
\begin{aligned}
\Phi_{11}(a,b)>0\,, \quad \Phi_{111}(a,b)<0\,, \quad P(a,b)>0
\end{aligned}
\end{equation}
which together with \eqref{SSYSSP} imply $\dot a > 0$, $\dot b < 0$ and
\begin{equation*}
    \frac{d}{ds}(a-b) = \frac{1}{sQ} (a-b)\big[ (d-1)P-Q \big] > 0.
\end{equation*}
Thus,  \eqref{MONT1},  \eqref{MONT2} and \eqref{QSIGN} must hold for $s\in(0,\eps]$. The solution can be continued in that manner for $s > \eps$ so long as $Q(s)<0$, \, $0<a(s)<\infty$ , and $0<b(s) < \infty$
on a maximal interval of existence $[0,T)$, with $T \leq \infty$. It is also clear that the solution cannot hit the diagonal
$a=b$ unless $Q = 0$. On the interval $(\eps, T)$ we clearly have
\begin{equation}
\label{abbound}
a(\eps) < a(s) < b(s) < b(\eps) \, , \quad s \in (\eps, T) \, ,
\end{equation}
that is $a(s)$ and  $b(s)$ stay
away from zero and in a bounded range as $s$ increases. Then \eqref{abbound} and the fact that $Q(s) = s^2 - \Phi_{11}(a,b)<0$, $s\in(0,T)$
imply that \, $T<\infty$ and
\begin{equation}
\label{Qsub}
\limsup_{s \to T_-} \SP Q (a(s) , b(s) , s) = 0.
\end{equation}
Also, in view of \eqref{MONT1} and \eqref{MONT2} we have $a-b \to c_0 \le 0$ as $s \to T_-$.

A computation shows that
$$
\frac{1}{2} \frac{d}{ds} Q^2 = 2 s Q - \frac{d-1}{s} (a-b)  \Big [ P(a,b) \Phi_{111}(a,b) +  \Phi_{112}(a,b)  Q \Big ]
$$
Using \eqref{QDEFSP} and the bounds \eqref{abbound} we see that for $s \in (\eps, T)$
\begin{equation}
\label{derivb}
\Big | \frac{d}{ds} Q^2  \Big | \le C \quad s \in (\eps, T) \, .
\end{equation}
Now if $T < \infty$ and
$$
\liminf_{s \to T_-} \SP Q (a(s) , b(s) , s)  <   \limsup_{s \to T_-} \SP Q (a(s) , b(s) , s) = 0
$$
then \eqref{derivb} would be violated. We thus deduce that $Q(s) \to 0$ as $s \to T_{-}$.
\end{proof}

\subsection{ Connection to a uniformly deformed state.}
 \label{secSHKCONNCT}

Next,  the smooth cavitating  solution $(\dot \varphi,\tfrac{\varphi}{s})$ of  \eqref{SSYSSP} constructed in Theorem \ref{CAVSOLEXISTSP}
will be connected to a uniformly deformed state through an outgoing shock of speed $\sigma > 0$. At the connection,
the Rankine-Hugoniot relations \eqref{SSJCOND} must be satisfied and they imply that
\begin{equation*}
    \varphi_{-}(\sigma) = \varphi_{+}(\sigma),
    \quad \sigma^2 \big( \avar_{+}(\sigma)-\avar_{-}(\sigma) \big) =
    \Phi_{1}(\avar_{+},\tfrac{\varphi_{+}}{\sigma})-
    \Phi_{1}(\avar_{-},\tfrac{\varphi_{-}}{\sigma})
\end{equation*}
where $\varphi_{-}(s) \equiv \varphi (s)$ is the given cavitating solution on the left, and $\varphi_{+} (s) = \lambda s$, for some $\lambda$, is the uniformly
deformed state on the right.

Thus to connect the cavitating solution $\varphi(s)$ to a uniformly  deformed state through a shock wave it suffices to define the function
\begin{equation}\label{SHKCOND}
p(s) := \frac{ \Phi_{1}(\avar,\bvart) -  \Phi_{1}(\bvart,\bvart)}{ \avar-\bvar} - s^2
\end{equation}
and to identify a zero of $p(s)$ for $s \in (0,T)$ the maximal interval of existence of $\varphi$. It is expedient to view the right hand side
of \eqref{SHKCOND} as a function of the principal stretches,
\begin{equation}
\begin{aligned}
R(a,b,s):&=
\begin{cases}
\frac{\Phi_{1}(a,b)-
    \Phi_{1}(b,b)}{a-b} - s^2, \, &a<b\\
\Phi_{11}(b,b)-s^2, \, &a=b\,,\\
\end{cases}
\end{aligned} \label{SHOCKR}
\end{equation}
and to observe that $R(a,b,s)$ is  continuous on $\big\{(a,b)\in \RR^2: 0 < a \leq b \big\} \times \RR$.

By  \eqref{MONT2} and \eqref{QABLIM},  the solutions $(v,\varphi)$ in Theorem \ref{CAVSOLEXISTSP} have well-defined limits
\begin{equation}\label{ABDEF}
0 <  A:= \lim_{s \to T_-} \dot{\varphi}(s) \, \leq \, B:=\lim_{s \to T_-} \Big(\frac{\varphi}{s}\Big) \, < \infty \, .
\end{equation}

\begin{theorem}[\bf existence of connection point]\label{EXISTCPT}
Assume $d\geq 2$ and $\Phi$ satisfies \eqref{spseis}-\eqref{DHATINF}. Let $(v,\varphi)$ be the solution to \eqref{ODEIVPSYS} constructed in Theorem \ref{CAVSOLEXISTSP} and $(a,b)=(\dot{\varphi},\frac{\varphi}{s})$. Let $A,B$ be the limits defined in \eqref{ABDEF}. 
Then,
\begin{itemize}

\item[$(i)$] if  $A=B$ then $(a, b)$ is connected continuously at  $\sigma=T$ to
the state $(B,B)$ associated with the uniform deformation $\varphi_+(s) = Bs$\,;

\item[$(ii)$] if $A<B$ then  $(a,b)$ is connected to a uniformly deformed
state through a Lax-admissible shock at some intermediate point $\sigma\in(0,T)$.

\end{itemize}
\end{theorem}

\begin{proof}
Recalling \eqref{SHOCKR}, we set for $s\in(0,T)$
\begin{equation}\label{pDEF}
\begin{aligned}
\!    p(s)&:=R(\avar,\bvart,s)\\[3pt]
 &= \bigg(\frac{g'(\avar)-g'(\bvar)}{\avar-\bvart}\bigg) + \Big(\bvar\Big)^{2d-2}
 \bigg(\frac{h'(\avar (\bvart)^{d-1})-h'((\bvart)^d)}{\avar (\bvart)^{d-1}-(\bvart)^d}\bigg) -
 s^2\,.
\end{aligned}
\end{equation}
 We observe that by \eqref{GHPROP1}, \eqref{DHATINF}, \eqref{PHIATZERO}, \eqref{MONT2}-\eqref{QABLIM}, \eqref{ABDEF} and \eqref{pDEF}:
\begin{itemize}
\item [$(1)$] As $s \to 0+$
\begin{equation}\label{pATZERO}
    p(s) \; \sim  \; \Big(\frac{s}{\varphi}\Big) \SP g'(\bvart) + \Big( \bvar \Big)^{d-2} h'((\bvart)^d)
    \, \to \, +\infty\, .
\end{equation}

\item [$(2)$] As $s \to T_{-}$
\begin{align}
p(s) \, &\to \, R(A,B,T)\label{pLIM}\\
Q(s) \, &\to \, T^2 - \Phi_{11}(A,B) = 0 \,.\label{QLIM}
\end{align}
\end{itemize}

\par\smallskip

Now, denote $p(T):= \lim_{s\to T_{-}} p(s) = R(A,B,T).$
If $A=B$, then  by \eqref{QLIM}
\begin{equation*}
p(T)= R(B,B,T) = \Phi_{11}(B,B) - T^2 = 0 \,.
\end{equation*}
On the other hand, if  $A<B$, then \eqref{GHPROP2} implies $\Phi_{111}<0$ and by \eqref{pLIM}, \eqref{QLIM}
\begin{equation*}
\begin{aligned}
p(T)=R(A,B,T) & \,= \, \frac{\Phi_{1}(A,B)-\Phi_{1}(B,B)}{A-B} - T^2 \\
         & \,= \, \Phi_{11}(C^*,B)-\Phi_{11}(A,B) \, <\, 0
\end{aligned}
\end{equation*}
for some $C^*\in(A,B)$. Then \eqref{pATZERO} implies that there exists $\sigma\in(0,T)$
such that $p(\sigma)=0$.
\end{proof}

\section{Necessity and uniqueness of the precursor shock} \label{necshock}

 In this section we show that there is uniqueness within the method of construction of cavitating solutions and that (when $d=2,3$) the connection must happen through a Lax-admissible shock.

\par\smallskip

We first prove that there exists at most one point in $(0,T]$ where $p(s)$ (defined in \eqref{pDEF}) vanishes and thus a unique connection (of the cavitating solution with a uniform deformation) occurs.

\begin{theorem}\label{UNIQCPT}
 Let $d\ge 2$ and $\Phi$ satisfy \eqref{spseis}-\eqref{GHPROP2}.
In Theorem \ref{EXISTCPT} there exists a unique  $\sigma \in(0,T]$ satisfying $p(\sigma) = 0$.
\begin{itemize}
\item[(i)]
If $\sigma<T$ then the connection happens via a Lax admissible shock.
\item[(ii)]
If $\sigma=T$ then the connection occurs via a sonic singularity at which $A=B$.
\end{itemize}
\end{theorem}

\begin{proof}\vspace{3pt}
Consider $p(s)$ for $s\in(0,T)$. Multiply \eqref{SHKCOND} by $(a-b)$ and then differentiate
to get
\begin{equation*}
\begin{aligned}
\dot{p}(a-b) + p (\dot{a}-\dot{b}) &=
 -\big(s^2-\Phi_{11}(a,b)\big) \dot{a}\\
& \quad +(d-1)\big(\Phi_{12}(a,b)-\Phi_{12}(b,b)\big)\dot{b}\\
 & \quad + \big(s^2 - \Phi_{11}(b,b)\big)\dot{b} -
2s (a-b)\,.
\end{aligned}
\end{equation*}
Recall that
\begin{equation}\label{ODETMP}
\begin{aligned}
    \big(s^2-\Phi_{11}(a,b)\big)\dot{a} = (d-1)P(a,b)\dot{b}, \quad
                        \dot{b} = \frac{1}{s}(a-b), \quad s\in(0,T)
    \end{aligned}
\end{equation}
and hence
\begin{equation*}
\begin{aligned}
\dot{p}(a-b) &+ p (\dot{a}-\dot{b})  \\
&=-(d-1) \Big[ P(a,b) -\Phi_{12}(a,b)+\Phi_{12}(b,b)\Big] \dot{b} - \big(s^2 +
\Phi_{11}(b,b)\big)\dot{b}\,.
\end{aligned}
\end{equation*}
Then, divide the result by $(a-b)<0$ and use \eqref{ODETMP} to conclude for
$s\in(0,T)$
\begin{equation}\label{DpSIGN}
\begin{aligned}
&\dot{p} + p \bigg( \frac{\dot{a} - \dot{b} }{a-b} \bigg) =\\[0pt]
&=-\frac{(d-1)}{s} \Bigg[\frac{\Phi_1(a,b)-\Phi_2(a,b)}{a-b} +\Phi_{12}(b,b) \bigg]
 - \big(s^2 + \Phi_{11}(b,b)\big)\frac{1}{s} <  0\,.
\end{aligned}
\end{equation}
To obtain the sign on the left hand side of \eqref{DpSIGN} we note that by  \eqref{spseis} and \eqref{GHPROP1}
\begin{equation*}
\begin{aligned}
&\frac{\Phi_{1}(a,b)-\Phi_2(a,b)}{a-b} + \Phi_{12}(b,b) \\ &\qquad \qquad = \,\frac{g'(a)-g'(b)}{a-b} +b^{d-2}\big(h'(b^d)- h'(ab^{d-1})\big) +
b^{2d-2}
h''(b^d)\\[2pt]
&\qquad\qquad = \, g''(c^*) +b^{2d-3}h''(v^*)\big(b - a\big) + b^{2d-2} h''(b^d)  \, >\, 0
\end{aligned}
\end{equation*}
for some $c^* \in (a,b)$ and $v^*\in (ab^{d-1},b^d)$.

Let now $\sigma\in(0,T]$ satisfying $p(\sigma)=0$. There exists at least one such point
according to Theorem \ref{EXISTCPT}.  On the other hand, \eqref{DpSIGN} implies that
at any $\sigma \in (0,T)$ where $p(s)$ vanishes we have that $\dot p (\sigma) < 0$.
Therefore, there can be at most one such point in $(0,T]$.

The discussion in section \ref{secselfsimilar}(c) indicates that the shock will satisfy
the Lax shock admissibility condition if $\Phi_{111} < 0$. For stored energies of the type \eqref{spseis}
this amounts to condition \eqref{GHPROP2}.
\end{proof}

\par\smallskip

We next show that for $d\in\{2,3\}$ the nontrivial solution $(v,\varphi)$ in Theorem \ref{CAVSOLEXISTSP} {\it cannot} be connected continuously to a uniformly deformed state. In view of Theorem \ref{UNIQCPT}, this means that the cavitating solution
is always associated with a unique Lax-admissible precursor shock for dimensions $d=2,3$. (The method of proof breaks down for $d \ge 4$.)

\begin{theorem}\label{MAINRESULTD3}  Assume $d \in\{2,3\}$ and $\Phi$ satisfies \eqref{spseis}-\eqref{DHATINF}. Let $(v,\varphi)$ be the solution to \eqref{ODEIVPSYS} constructed in Theorem \ref{CAVSOLEXISTSP} and let $(a,b)=(\dot{\varphi},\frac{\varphi}{s})$. Let $A,B$ be the limits defined in \eqref{ABDEF}. Then $A<B$ and, as a   consequence,
\begin{itemize}

\item[$(i)$]  $(a,b)$, the solution to \eqref{SSYSSP}, cannot be connected continuously to a uniformly deformed state at any $s\in(0,T]$\,;

\item[$(ii)$] \SP there exists unique point $\sigma\in(0,T)$ at which the solution $(a,b)$ can be connected to a uniformly deformed state associated with ${\varphi}_+(s)= b(\sigma)s $ through a Lax-admissible shock.
\end{itemize}
\end{theorem}
\begin{proof}
The proof utilizes the properties \eqref{PLIM} and \eqref{PMPHI11LIM} of the function $P(a,b)$ proved  in the appendix.
Since $Q<0$, $s\in(0,T)$, we rewrite \eqref{SSYSSP} in the equivalent form
\begin{equation}\label{ODEH}
\begin{aligned}
    \dot{a} = \frac{(d-1)}{s}P(a,b)H(s), \quad
                        \dot{b} = \frac{1}{s}(a-b) \, ,
                            \end{aligned}
\end{equation}
where the function $H(s)$ is defined by
\begin{equation}\label{HDEF}
H := \frac{a-b}{Q} = \frac{a-b}{s^2 - \Phi_{11}(a,b)} \,, \quad s\in(0,T).
\end{equation}

To prove that $A<B$, we will argue by contradiction. Suppose that
\begin{equation}\label{ABEQ}
 A = B =: \lambda
\end{equation}

\smallskip
\noindent
 {\it Step 1.} We will prove that
\eqref{ABEQ} implies that there exists $\eps>0$ such that
\begin{equation}\label{HMONPROP}
    H(s)>0, \quad \frac{d}{ds}\bigl( s^3 H(s)\bigr) > 0 \,, \quad
    s\in(T-\eps,T) \,.
\end{equation}
and, as a consequence, we must have {\it either}
\begin{equation}\label{HLIM}
\begin{aligned}
&\qquad\lim_{s\to T^-} H(s) = c \quad \mbox{for some} \quad 0<c<\infty \\[-3pt]
&\mbox{\it or}\\[-3pt]
&\qquad\lim_{s\to T^-} H(s) = +\infty.
\end{aligned}
\end{equation}

Indeed, by \eqref{QSIGN} we have $H(s)>0$ for $s\in(0,T)$. Next, by
\eqref{QDEFSP}, \eqref{ODEH} we have
\begin{equation}\label{DQ}
\begin{aligned}
\dot{Q} & \, = \, 2s - \big[ \Phi_{111}(a,b)\dot{a} + (d-1)\Phi_{112}\dot{b}\Big]\\
        &\, = \, 2s -(d-1) H \big[\Phi_{111}(a,b) P(a,b)+\Phi_{112}(a,b)Q\big]\frac{1}{s}\\
        &\, = \,  \big[ Q+\Phi_{11}(a,b)\big]\frac{2}{s}  - (d-1) H \big[\Phi_{111}(a,b) P(a,b) + \Phi_{112}(a,b)Q
        \big]\frac{1}{s}
\end{aligned}
\end{equation}
and  using \eqref{ODEH}, \eqref{HDEF} we obtain
\begin{equation*}
\begin{aligned}
 \frac{dH}{ds}  &= \frac{\dot{a}-\dot{b}}{Q} - \frac{\dot{Q}(a-b)}{Q^2}\\[2pt]
& =\frac{(d-1)}{s} \frac{H}{Q}P(a,b) - \frac{H}{s} - \frac{H}{Q}\dot{Q}\\[2pt]
& = (d-1)\frac{H}{Q} \big[ P(a,b) - \frac{2}{d-1}\Phi_{11}(a,b)\big]\frac{1}{s} \\[2pt]
&\,\quad + (d-1)\frac{H^2}{Q} \big[\Phi_{111}(a,b)P(a,b) + \Phi_{112}(a,b)Q
\big]\frac{1}{s} -\frac{3H}{s}
\end{aligned}
\end{equation*}
Rearranging the terms in the resulting expression, we obtain
\begin{equation}\label{DH}
\begin{aligned}
&\frac{1}{s^2(d-1)}\frac{d}{ds}  \bigl(s^3 H \bigr)= \\[2pt]
&\quad\qquad  =  \frac{H}{Q} \big[ P(a,b) -
\frac{2}{d-1}\Phi_{11}(a,b)\big]   + \frac{H^2}{Q} \big[\Phi_{111}(a,b)P(a,b) + \Phi_{112}(a,b)Q \big]\\[3pt]
&\quad\qquad  = H^2 \bigg(\frac{P(a,b)-\Phi_{11}(a,b)}{a-b} + \Phi_{112}(a,b)\\
& \quad\qquad \qquad\qquad + \frac{1}{Q} \Big[\Phi_{111}(a,b)P(a,b)- \frac{1}{H}\Bigl(\frac{3-d}{d-1}\Bigr)\Phi_{11}(a,b) \Big]\biggr) \SP.\\
\end{aligned}
\end{equation}

\par\smallskip

Now, observe that \eqref{PLIM}, \eqref{ABDEF}, and the assumption \eqref{ABEQ}
imply
\begin{equation}\label{INFPART}
\begin{aligned}
\lim_{s\to T_{-}} \big( \Phi_{111}(a,b)P(a,b) \big)=
\Phi_{111}(\lambda,\lambda)\Phi_{11}(\lambda,\lambda)  <  0 \,.
\end{aligned}
\end{equation}
Recall that $H>0$. Hence by \eqref{QSIGN}, \eqref{QABLIM}$_1$, \eqref{DPHISIGN1}$_1$, \eqref{INFPART}
and $d\in\{2,3\}$ we have
\begin{equation}\label{INFPARTLIM}
 \lim_{s\to T_{-}} \frac{1}{Q} \Big[\Phi_{111}(a,b)P(a,b)- \frac{1}{H}\Bigl(\frac{3-d}{d-1}\Bigr)\Phi_{11}(a,b) \Big] = +\infty.
\end{equation}
Also, by \eqref{ABDEF}, \eqref{ABEQ}, and \eqref{PMPHI11LIM}
\begin{equation}\label{FINITEPART}
\begin{aligned}
  \lim_{s\to T_{-}} \bigg( \frac{P(a,b)-\Phi_{11}(a,b)}{a-b}& + \Phi_{112}(a,b) \bigg)
  =  \frac{1}{2} \Big(3\Phi_{112}(\lambda,\lambda) - \Phi_{111}(\lambda,\lambda)\Big)\,.
     \end{aligned}
\end{equation}

\par\smallskip

Finally, combining \eqref{INFPARTLIM}, \eqref{FINITEPART}, we conclude that there exists
$\eps>0$ such that the right-hand side of \eqref{DH} is positive for all
$s\in(T-\eps,T)$. This establishes \eqref{HMONPROP} and as a consequence \eqref{HLIM}.

\medskip
\noindent
{\it Step 2.}
By \eqref{QABLIM}$_1$ and \eqref{ABEQ}, both the numerator and
the denominator of $H(s)$ satisfy
\begin{equation}\label{HNDLIM}
    \lim_{s\to T_{-}} \bigl(a(s) - b(s) \bigr)=\lim_{s\to T_{-}} Q(s)=0  \, .
\end{equation}
This motivates to consider the ratio
\begin{equation}\label{rDEF}
\begin{aligned}
r(s) &:= \frac{\frac{d}{ds}Q}{\frac{d}{ds}( a - b )}
\\
&\,=\frac{2s -(d-1) H \big[\Phi_{111}(a,b)
P(a,b)+\Phi_{112}(a,b)Q\big]\frac{1}{s}}{(d-1)HP(a,b)\frac{1}{s}-(a-b)\frac{1}{s}}
\\[3pt]
&\,=\frac{ \frac{2}{d-1}\frac{s^2}{H} -
\big[\Phi_{111}(a,b)P(a,b)+\Phi_{112}(a,b)Q\big]}{P(a,b)-\frac{1}{d-1}Q}   \, ,
\end{aligned}
\end{equation}
(where we used \eqref{ODEH}, \eqref{HDEF}, \eqref{DQ})  and to  look for the limiting value of $r(s)$ as
$s\to T_{-}$. To this end we will use \eqref{HLIM} and consider two separate cases:

\par\medskip

{\it Case $1$.} Suppose that \eqref{HLIM}$_1$ holds. Observe that by
\eqref{QABLIM}$_1$
\begin{equation*}\label{QLIMMD}
0=\lim_{s\to T_{-}}Q(s)=\Phi_{11}(\lambda,\lambda)-T^{2}
\end{equation*}
and hence by \eqref{PLIM} and \eqref{ABEQ}
\begin{equation*}
\lim_{s\to T_{-}} P(a,b)= T^2\,.
\end{equation*}
Thus, \eqref{QABLIM}$_1$, \eqref{ABDEF}, \eqref{ABEQ}, and \eqref{INFPART} imply
\begin{equation*}
    \lim_{s \to T^-} r(s) =
   \frac{2}{d-1} \frac{1}{c}-\Phi_{111}(\lambda,\lambda) \, .
\end{equation*}
On the other hand, by l'H\^opital rule, we have
\begin{equation*}
0< \frac{1}{c} = \lim_{s \to T_{-}}
\frac{1}{H(s)} = \lim_{s \to T_{-}} r(s) =  \frac{2}{d-1} \frac{1}{c}-\Phi_{111}(\lambda,\lambda)
\end{equation*}
Since $c\in(0,\infty)$ and  $d\in\{2,3\}$
this leads to a contradiction
\begin{equation*}
0 \SP <  -\SP\Phi_{111}(\lambda,\lambda) \SP c = \biggl(\frac{d-3}{d-1}\biggr) \leq
0\,.
\end{equation*}
Hence \eqref{HLIM}$_1$ cannot hold.

\par\medskip

{\it Case $2$.} Suppose that \eqref{HLIM}$_2$ holds. Then \eqref{QABLIM}$_1$, \eqref{ABDEF} and
\eqref{ABEQ} imply
\begin{equation}\label{rLIM2}
    \lim_{s \to T^-} r(s) = -\Phi_{111}(\lambda,\lambda) \,.
\end{equation}
By \eqref{DPHISIGN1}$_2$, \eqref{HDEF}, \eqref{HLIM}$_2$, \eqref{HNDLIM},
\eqref{rDEF}, \eqref{rLIM2} and l'Hopital's rule we obtain
\begin{equation*}
0 = \lim_{s\to T_{-}} \frac{1}{H(s)} = \lim_{s\to T^-} r(s) = -\Phi_{111}(\lambda,\lambda)  > 0
\end{equation*}
which is again a contradiction. Thus, \eqref{HLIM}$_2$ cannot hold.

Since the hypothesis \eqref{ABEQ} leads to a contradiction, we conclude that $A < B$ for $d=2,3$ and the
proof is completed.
\end{proof}

\section {Rescaling, inner solution, and the dynamic bifurcation diagram}
\label{secbifurc}

The objective of this section is to construct the bifurcation diagram of the dynamically cavitating solution and to determine
the critical stretching at which cavitation occurs.
We first present an outline of the approach that we follow.

We use the cavitating solution constructed in
Theorem \ref{EXISTCPT}  under hypotheses \eqref{spseis}--\eqref{GGROWTH} and hypothesis \eqref{DHATINF} (relating to the boundary data
at the cavity). Recall that this solution, denoted by $(\varphi, v)(s \,;\varphi_0, v_0)$,  depends on two parameters $\varphi_0 > 0$ and $v_0$ in
\eqref{ODEIVPSYS}, and it is defined on a maximal
interval of existence $[0,T)$ with $T=T(\varphi_0,v_0)<\infty$.  By Theorem \ref{UNIQCPT} for every pair $\varphi_0,v_0>0$ there exists a unique
point
\begin{equation}\label{SIGMAMAPDEF}
\sigma=\sigma(\varphi_0,v_0) \in (0,T(\varphi_0,v_0)] \quad \mbox{that satisfies} \quad
p(\sigma;\varphi_0,v_0)=0\,,
\end{equation}
where $p$ is defined in \eqref{SHKCOND} and the identity $p(\sigma) = 0$ corresponds to the Rankine-Hugoniot jump
conditions.  We recall that if $\sigma(\varphi_0,v_0) < T (\varphi_0,v_0)$ then the connection of the cavitating solution to the uniformly deformed
state on the right happens through a shock, while if $\sigma(\varphi_0,v_0) = T (\varphi_0,v_0)$ then the cavitating solution
connects to a uniform deformation in a $C^1$ fashion through a sonic singularity. Also, that the latter possibility is excluded in Theorem \ref{MAINRESULTD3} for dimensions $d=2, 3$, but it might conceivably occur for higher dimensions. In both cases, the transversal principal stretch
$\frac{\varphi(\sigma)}{\sigma}$ at the shock (or sonic singularity) coincides with the value of external stretching associated to the forming cavity.

In view of the above, we define the mapping  $\Lambda(\varphi_0,v_0) : \RR^2_{+} \to \RR$ by
\begin{equation}\label{STRETCHDEF}
\begin{aligned}
\Lambda(\varphi_0,v_0) =  \frac{\varphi(\sigma(\varphi_0,v_0)\SP; \varphi_0,v_0)}{\sigma(\varphi_0,v_0)}.
\end{aligned}
\end{equation}
Referring to the discussion of Section \ref{secCauchy}, we recall that at the cavity the specific volume $v_0$ is connected with
the cavity velocity $\varphi_0$ via the relation
$v_0 = V ( \varphi_0)$,
where $V:[0,\infty) \to (0,\infty)$ is a continuous function that encodes the boundary condition at the cavity and has the form
\begin{equation}
\label{CSTRESSBNDV1}
V(x) :=
\begin{cases}
H = {h'}^{-1} (0) &\mbox{ for stress free cavities},  \\
h'^{-1}(G(x))   &\mbox{ for cavities with content}. \\
\end{cases}
\end{equation}
The diagram of the map
$$
\varphi_0 \mapsto \Lambda(\varphi_0, V(\varphi_0) )
$$
is precisely the bifurcation diagram of a solution with cavity.
{A numerical computation of this diagram appears in Figure \ref{DYNSTAT}; for explanation of numerical results see Appendix \ref{app3}. }
In the sequel, we study various analytical features of the bifurcation diagram. We are particularly interested in determining
the critical stretching $\lambda_{cr}$ at which a cavity opens. This will be captured by the limit
\begin{equation}\label{BIFLIM}
\lim_{\varphi_0\to 0+}\Lambda(\varphi_0,V(\varphi_0))
\end{equation}
and one objective is to determine a formula for the computation of \eqref{BIFLIM} in the dynamic case.

\subsection{Rescaling}

Let $(\varphi,v)(s\SP; \varphi_0,v_0)$  be the solution to \eqref{ODEIVPSYS}, with $\varphi_0,v_0>0$. Numerical experiments indicate
that the functions $\{v(s \SP ;\varphi_0,V(\varphi_0))\}$,
with $V$ defined by \eqref{CSTRESSBNDV1},
converge pointwise to a step function with a jump located at $s=0$ as $\varphi_0 \to 0$; see Figure \ref{DETPIC}. This indicates
an inner layer with respect to $\varphi_0$ and points to resolving the jump as the key in determining the limit \eqref{BIFLIM}.
\begin{figure}[t] \centering
\includegraphics*[width=4.5in, height=3.5 in]{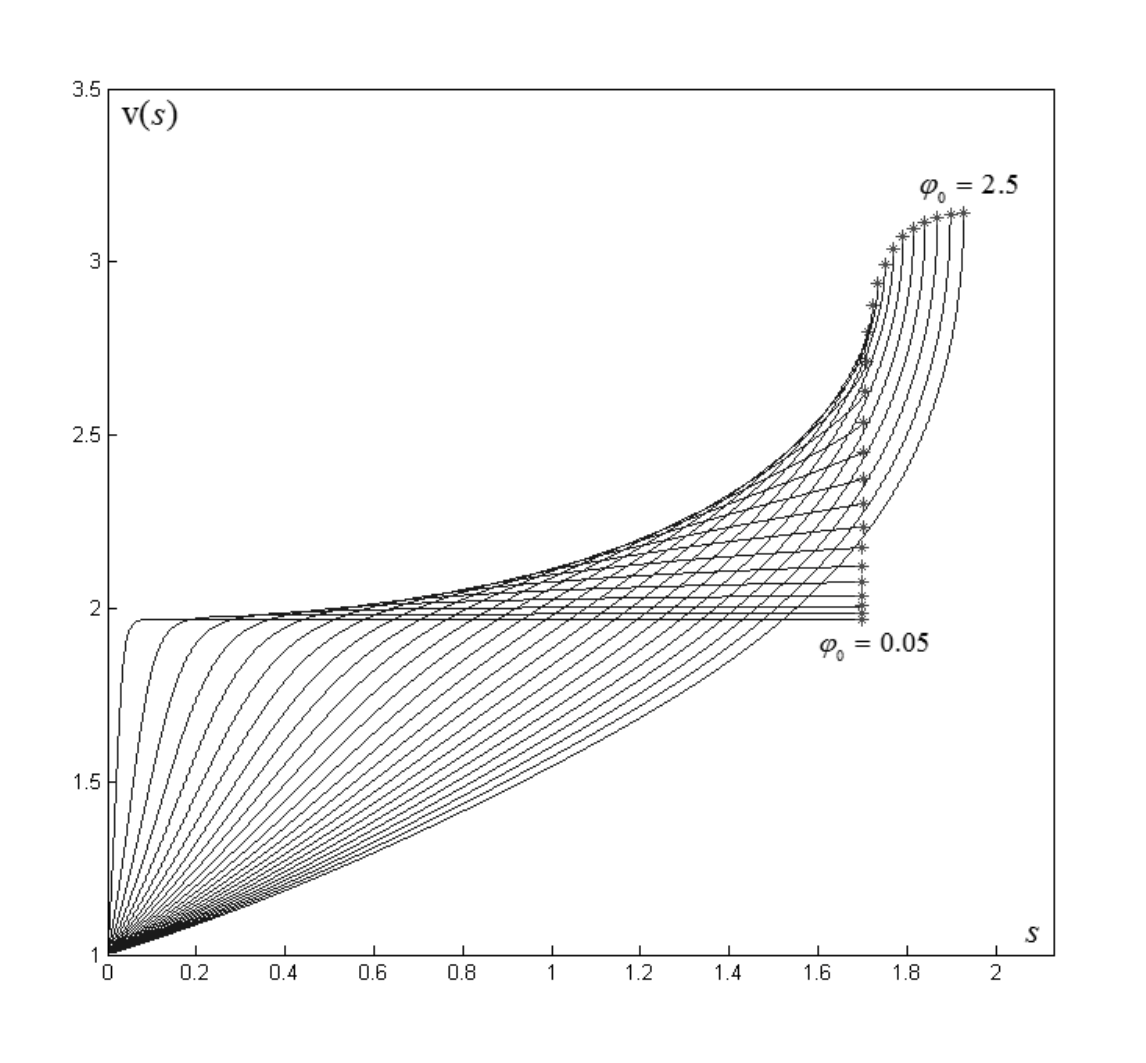}
\caption{
$v(s; \varphi_0, H)$ with $\varphi_0\in[0.05,2.5]$,
$g(x)=\tfrac{1}{2}x^2$, $h(x)=(x-1)\ln(x)$  and stress free cavity.
} \label{DETPIC}
\end{figure}

To capture the behavior near the origin we rescale $(\varphi, v )(s\SP; \varphi_0,v_0)$ with respect to the initial value $\varphi_0>0$
using the scaling transformation
\begin{equation}\label{RESCSOL}
\begin{aligned}
    \psi(\xi; \varphi_0, v_0) := \frac{\varphi(\varphi_0\xi\SP; \varphi_0,
    v_0)}{\varphi_0},\quad
    \delta(\xi; \varphi_0, v_0) := v(\varphi_0\xi\SP; \varphi_0,v_0)\,.
\end{aligned}
\end{equation}
Note that $(\psi, \delta)$ satisfy
\begin{equation*}
\begin{aligned}
\frac{\psi(\xi)}{\xi} = \frac{\varphi(s)}{s}\Bigr|_{s=\varphi_0\xi}\,, \quad
{\psi'}(\xi) = \dot\varphi(s)\Bigr|_{s=\varphi_0 \xi}\,, \quad
\delta(\xi)&=v(s)\Bigr|_{s=\varphi_0 \xi}
\end{aligned}
\end{equation*}
with $'$ denoting the differentiation with respect to $\xi=\frac{s}{\varphi_0}$.  The rescaled function $(\psi, \delta)$
is now defined on the maximal interval of existence $[0,\mathcal{T}(\varphi_0,v_0))$, where
$\mathcal{T}(\varphi_0,v_0) := \frac{1}{\varphi_0} T(\varphi_0,v_0) < \infty$ and satisfies
the initial value problem
\begin{equation}\label{ODEIVPSYSRESC}
\left\{
\begin{aligned}
\psi'(\xi) &= \delta\Big(\frac{\xi}{\psi}\Big)^{d-1}\\
\delta'(\xi) & = \bigg(\frac{d-1}{\psi}\bigg) \frac{ \big(
\frac{\xi}{\psi}\big)^{2d-3}\delta \Big( \delta \big(\frac{\xi}{\psi}\big)^{d} - 1
\Big)\Big[\xi^2 \varphi_0^2  - g''( \delta (\tfrac{\xi}{\psi})^{d-1})
\Big]}{\Bigl[-h''(\delta)+\big(\xi^2 \varphi_0^2
-g''(\delta(\tfrac{\xi}{\psi})^{d-1})\big)\big(\frac{\xi}{\psi}\big)^{2d-2}\Bigr]
}\\[2pt]
&\qquad +\bigg(\frac{d-1}{\psi}\bigg)\frac{ \big( \frac{\xi}{\psi}\big)^{d-2} \Big[g'(
\delta (\tfrac{\xi}{\psi})^{d-1})-g'(\frac{\psi}{\xi})
\Big]}{\Bigl[-h''(\delta)+\big(\xi^2 \varphi_0^2
-g''(\delta(\tfrac{\xi}{\psi})^{d-1})\big)
\big(\frac{\xi}{\psi}\big)^{2d-2}\Bigr]\SP}\\
&\begin{aligned}
\psi(0 \SP; \varphi_0,v_0)&=1\\
\delta(0\SP; \varphi_0,v_0)&=v_0
\end{aligned}
\end{aligned}\right.
\end{equation}
As \eqref{ODEIVPSYS} is equivalent to \eqref{ODESP}, the rescaled function $\psi(\xi; \varphi_0, v_0)$ will also satisfy the
second order differential equation
\begin{equation}\label{RESCODE}
\begin{aligned}
\Big(\varphi_0^2 \xi^2-\Phi_{11}(\psi',\tfrac{\psi}{\xi})\Big) \psi'' =
\frac{d-1}{\xi}\Big(\psi' -
\frac{\psi}{\xi}\Big)P(\psi',\tfrac{\psi}{\xi})\,.
\end{aligned}
\end{equation}
Theorem \ref{CAVSOLEXISTSP} now gives:

\begin{lemma}\label{resccavexist}
Let $d\geq 2$ and $\Phi$ satisfy \eqref{spseis}-\eqref{GGROWTH}. The rescaled
function $(\psi, \delta)$ in \eqref{RESCSOL}
is defined on the maximal interval of existence $[0,\mathcal{T}(\varphi_0,v_0))$, with
$\mathcal{T}(\varphi_0,v_0) := \frac{1}{\varphi_0} T(\varphi_0,v_0)$, and satisfies
\eqref{ODEIVPSYSRESC}, while $\psi\in C^2(0,\mathcal{T})$ solves \eqref{RESCODE}  on $(0,\mathcal{T})$.
Moreover,
\begin{itemize}
\item [$(i)$] \; $\psi'$, $\tfrac{\psi}{\xi}$, $(\psi' -\tfrac{\psi}{\xi})$ are
strictly monotonic and satisfy
\begin{equation}\label{MONTR1}
\psi''(\xi)>0, \quad \frac{d}{d\xi}\Big(\frac{\psi}{\xi}\Big) < 0, \quad
\frac{d}{d\xi}\Big(\psi' - \tfrac{\psi}{\xi}\Big) > 0  \quad \mbox{on} \quad
0<\xi<\mathcal{T}.
\end{equation}

\item[{$(ii)$}] \; For each \SP $0<\xi<\tau<\mathcal{T}$
\begin{equation}\label{MONTR2}
0 <  {\psi'}(\xi)  <  {\psi'}(\tau) < \frac{\psi(\tau)}{\tau}< \frac{\psi(\xi)}{\xi}\,, \quad Q({\psi'},\tfrac{\psi}{\xi},\xi
\varphi_0) <0\,.
\end{equation}

\item[$(iii)$] \; The following limits exists and satisfy \vspace{6pt}
\begin{equation*}
\begin{aligned}
\lim_{\xi\to \mathcal{T}_-}Q({\psi'},\tfrac{\psi}{\xi},\xi \varphi_0)=0, \,\quad
\lim_{\xi\to \mathcal{T}_-} {\psi'} = A \leq \lim_{\xi\to \mathcal{T}_-} \Big( \frac{\psi}{\xi} \Big)  = B  \\
\end{aligned}
\end{equation*}
for some $A,B>0$.
\end{itemize}
\end{lemma}

Due to the form of \eqref{RESCODE} the rescaling leads to a regular perturbation problem and it is expected
that the limit has  a globally defined solution. Below, we study this limiting process.

\subsection{Uniform bounds for the rescaled solutions}

In the sequel we impose the condition that for some $\nu > 0$
\begin{equation}\tag{H5}\label{D2HGROWTH}
   \Phi_{11} (x, x) =  g''(x) + x^{2d -2}  h''( x^d ) \ge \nu^2 > 0  \, , \quad \mbox{for $0 < x < \infty$}.
\end{equation}
This condition is fulfilled  if $g$ satisfies $g'' (x) \ge \nu^2 > 0$. The latter assumption is possible for $d \ge 3$,
but it is inconsistent with \eqref{GGROWTH} for $d =2$. An alternative is to impose at infinity the condition
\begin{equation}\tag{H5$^\prime$}\label{D2pHGROWTH}
    \limsup_{x \to \infty } h''(x)x^{2-\frac{2}{d}} > 0
\end{equation}
which together with \eqref{GHPROP1} implies \eqref{D2HGROWTH}.
Hypothesis  \eqref{D2HGROWTH} ensures that $\mathcal{T}(\varphi_0,v_0) \to \infty$ as $\varphi_0 \to 0_+$.

\begin{lemma}\label{SPEEDNULMM}
Let $\Phi$ satisfy \eqref{spseis}-\eqref{GGROWTH}, \eqref{D2HGROWTH}. Let $T(\varphi_0,v_0), \mathcal{T}(\varphi_0,v_0)$ denote the times of existence defined in Theorem \ref{CAVSOLEXISTSP} and Lemma \ref{resccavexist}, respectively. Then,
\begin{align}
 T(\varphi_0,v_0)  = \varphi_0\mathcal{T}(\varphi_0,v_0)  & > \nu\,,  \quad \,\,\, \varphi_0,v_0>0\,.\label{MAXTIMEUPOS}
\end{align}
\end{lemma}

\begin{proof}
By \eqref{GHPROP2}, \eqref{DPHISIGN1} and \eqref{D2HGROWTH}  we get
\begin{equation}
\label{PHI11UPOS}
    \Phi_{11}(a,b) > \Phi_{11}(b,b)   > \SP \nu^2 \SP >\SP 0
\end{equation}
for all $0< a \leq b$.
Let $(\varphi,v)(s \SP; \varphi_0,v_0)$ be the solution of \eqref{ODEIVPSYS} and $A, B$ the limits defined in \eqref{ABDEF}. Then from \eqref{QDEFSP}, \eqref{QABLIM}, \eqref{ABDEF} and \eqref{RESCSOL} we conclude
\begin{equation*}
\begin{aligned}
\varphi_0 \mathcal{T}(\varphi_0,v_0) =  T(\varphi_0,v_0) = \sqrt{\Phi_{11}(A,B)}
\SP > \SP \nu\,.
\end{aligned}
\end{equation*}
\end{proof}

In the sequel we employ the notation
\begin{align}
f(\xi \SP; \varphi_0, v_0) & =
\Phi_{11}({\psi'},\tfrac{\psi}{\xi})\Big(\frac{\xi}{\psi}\Big)^{2d-2}\notag
 \\&= h''(\delta)+g''(\delta(\tfrac{\xi}{\psi})^{d-1})
\Big(\frac{\xi}{\psi}\Big)^{2d-2} \label{FDEF}
\\
\widehat{Q}(\xi \SP; \varphi_0, v_0) &= \xi^2 \varphi_0^2 - \Phi_{11}(\psi',\tfrac{\psi}{\xi})\notag
\\
&
 = \xi^2 \varphi_0^2  - \Big[ g''(\delta\big(\tfrac{\xi}{\psi}\big)^{d-1}) + \Big(
\frac{\psi}{\xi}\Big)^{2d-2} h''(\delta)\Big] \label{QHDEF}
\\
D(\xi \SP; \varphi_0,v_0) & =  \Big(\frac{\xi}{\psi}\Big)^{2d-3}
  \delta \Big(\delta \Big(\frac{\xi}{\psi} \Big)^d -1 \Big)
\Big(\xi^2 \varphi_0^2 - g''(\delta\big(\tfrac{\xi}{\psi}\big)^{d-1})\Big) \notag \\
& \, \quad + \Big(\frac{\xi}{\psi}\Big)^{d-2}
 \Big( g'(\delta\big(\tfrac{\xi}{\psi}\big)^{d-1}) - g'(\tfrac{\psi}{\xi})\Big)\,, \label{DDEF}
\end{align}
where $(\psi,\delta)(\xi \SP ; \varphi_0, v_0)$ is the solution of
\eqref{ODEIVPSYSRESC}. We note that under Hypothesis \eqref{GGROWTH}
\begin{align}
|g'(x) | \le \SP \widehat{\gamma} \SP \max(1,x^{d-2}) \, , \quad \mbox{for $x\in[0,\infty)$} \, .
\label{GAMMAHDEF}
\end{align}

The next lemma provides bounds on the Cauchy stress.

\begin{lemma}\label{CSTRESSBNDS}
 Assume $d\geq 2$ and $\Phi$ satisfies \eqref{spseis}-\eqref{GGROWTH}, \eqref{D2HGROWTH}. Let $(\psi,\delta)(\xi \SP;\varphi_0,v_0)$ be the solution to \eqref{ODEIVPSYSRESC} defined
on a maximal interval of existence $[0,\mathcal{T}(\varphi_0,v_0))$.  Suppose $\tau>0$ is fixed, $\nu >0$ satisfies \eqref{PHI11UPOS} and
\begin{equation}\label{EPSDEF}
\eps_{\tau}:=\nu \SP \big(2(1+\nu+\nu^{-1}+\tau)\big)^{-1}\,.
\end{equation}
Then, for every
\begin{equation*}
0<\varphi_0 < \eps_{\tau},  \quad 0<v_0<\infty \, ,
\end{equation*}
the interval $[0,\tau] \subset [0,\mathcal{T}(\varphi_0,v_0))$ and
\begin{equation}\label{RESCRADC}
\begin{aligned}
    \widehat{T}_{rad}(\xi \SP; \varphi_0,v_0) &:= \Big[\Phi_{1}({\psi'},\tfrac{\psi}{\xi})\Big(\frac{\xi}{\psi} \Big)^{d-1}\Big] (\xi \SP; \varphi_0,v_0 )\\
    & \; = \Big[h'(\delta)+g'(\delta(\tfrac{\xi}{\psi})^{d-1})\Big( \frac{\xi}{\psi} \Big)^{d-1}\Big](\xi \SP; \varphi_0,v_0 )
\end{aligned}
\end{equation}
satisfies
\begin{equation}\label{RESCRADCBND}
   \Big|\frac{d}{d\xi} \big(\widehat{T}_{rad}(\xi \SP; \varphi_0,v_0) \big) \Big| <
   c_{rad}(1+\tau^{d+1})\,,\quad \xi \in (0,\tau]
\end{equation}
with $c_{rad} :=   6(d-1)(1+g''(0)+\widehat{\gamma})$.
\end{lemma}

\begin{proof}
Take $\varphi_0\in(0,\eps_{\tau}),\,v_0\in(0,\infty)$. By \eqref{MAXTIMEUPOS} the maximal time
$\mathcal{T}(\varphi_0,v_0)>\frac{\nu}{\varphi_0}$  thus  enforcing that $[0,\tau] \SP \subset \SP
[0,\mathcal{T}(\varphi_0,v_0))$ and $\widehat{T}_{rad}(\xi \SP; \varphi_0,v_0)$ is
well-defined on $(0,\tau]$.

Next, we compute
\begin{equation}\label{COMPTR1}
\begin{aligned}
\frac{d}{d\xi} \big( \widehat{T}_{rad}(\xi \SP; \varphi_0,v_0) \big) 
&= \frac{d}{d\xi}
\Big(g'(\delta\big(\tfrac{\xi}{\psi}\big)^{d-1})\Big(\frac{\xi}{\psi}\Big)^{d-1}
+h'(\delta)\Big)\\
& = \Big( g''(\delta\big(\tfrac{\xi}{\psi}\big)^{d-1})\Big(\frac{\xi}{\psi}\Big)^{2d-2}
 + h''(\delta) \Big) {\delta'}\\
 & \,\quad +\frac{d-1}{\psi} \Big(\frac{\xi}{\psi}\Big)^{2d-3}
 \delta \Big(1-\delta \Big(\frac{\xi}{\psi} \Big)^d \Big) g''(\delta\big(\tfrac{\xi}{\psi}\big)^{d-1})\\
  & \,\quad
  +\frac{d-1}{\psi} \Big(\frac{\xi}{\psi}\Big)^{d-2}
   \Big(1-\delta \Big(\frac{\xi}{\psi} \Big)^d\Big)
 g'(\delta\big(\tfrac{\xi}{\psi}\big)^{d-1})\,.
\end{aligned}
\end{equation}
Observe that by \eqref{ODEIVPSYSRESC}$_2$, \eqref{DDEF}, and \eqref{QHDEF}
\begin{equation*}
\begin{aligned}
\Big( g''(\delta\big(\tfrac{\xi}{\psi}\big)^{d-1})\Big(\frac{\xi}{\psi}\Big)^{2d-2}
 + h''(\delta) \Big) {\delta'} = \frac{d-1}{\psi} D(\xi)\bigg(\frac{\xi^2
 \varphi_0^2}{\widehat{Q}(\xi)}-1\bigg)
 \end{aligned}
\end{equation*}
and hence \eqref{COMPTR1} reads
\begin{equation}\label{COMPTR2}
\begin{aligned}
\frac{1}{d-1} \SP  \frac{d}{d\xi}  & \big( \widehat{T}_{rad}(\xi \SP; \varphi_0,v_0) \big) \\
&\,=\Big(\frac{\xi}{\psi}\Big)^{2d-2}
  \delta \Big(1 - \delta \Big(\frac{\xi}{\psi} \Big)^d \Big)\SP
\xi \varphi_0^2 + \Big(\frac{\xi}{\psi}\Big)\frac{D(\xi)}{\widehat{Q}(\xi)}\SP \xi
\varphi_0^2\\
 & \,\,\quad +  \frac{1}{\psi}\Big(\frac{\xi}{\psi}\Big)^{d-2}
 \Big( g'(\tfrac{\psi}{\xi}) -
 \delta\Big(\frac{\xi}{\psi}\Big)^{d}g'(\delta\big(\tfrac{\xi}{\psi}\big)^{d-1})\Big) \\
 & \,\SP =: \SP I_1 + I_2 + I_3\,.
\end{aligned}
\end{equation}

We now estimate the right-hand side of \eqref{COMPTR2}. By \eqref{MONTR1}-\eqref{MONTR2}
\begin{equation}\label{PSIPROP}
    1 = \psi(0) < \psi(\xi), \,\quad 0 < \Big( \frac{{\psi'}\xi}{\psi} \Big) = \delta
    \Big( \frac{\xi}{\psi}\Big)^d < 1\,, \quad \xi\in(0,\mathcal{T})
\end{equation}
and hence
\begin{equation}\label{I1EST}
    0 \, \leq\, I_1 =
 \SP \Big(\frac{\xi}{\psi}\Big)^{2d-2} \delta \!
    \Big(1 - \delta \Big(\frac{\xi}{\psi}\Big)^d\Big)
    \xi \varphi_0^2   < \xi^{d-1} \varphi_0^2  \SP
\,\leq \, \tau^{d-1} \varphi_0^2\,,\quad \xi\in(0,\tau]\,.
\end{equation}

Next, by \eqref{GGROWTH}, \eqref{PSIPROP}$_2$, and \eqref{GAMMAHDEF} obtain
\begin{equation}\label{DGPRODBND}
\begin{aligned}
|I_3| &\le \Big|\Big(\frac{\xi}{\psi}\Big)^{d-2} g'(\tfrac{\psi}{\xi}) \Big| +
\Big|\Big(\frac{\xi}{\psi}\Big)^{d-2}\! g'(\delta (\tfrac{\xi}{\psi})^{d-1})\Big|
\\
&\quad = \Big|\Big(\frac{\xi}{\psi}\Big)^{d-2} g'(\tfrac{\psi}{\xi}) \Big| +
\Big|\Big(\frac{{\psi'}\xi}{\psi}\Big)^{d-2} \frac{g'({\psi'})
}{({\psi'})^{d-2}}\Big| \, \leq \, 2 \widehat{\gamma} \SP (1+\tau^{d-2})\,, \quad \xi\in(0,\tau]\,.
\end{aligned}
\end{equation}

{Now, using \eqref{PSIPROP}, we obtain the estimate
\begin{equation*}
\begin{aligned}
 \SP \Big|\Big(\frac{\xi}{\psi}\Big)^{2d-3} \delta \!
    \Big(\delta \Big(\frac{\xi}{\psi}\Big)^d-1\Big) \Big| & =
    \Big(    \frac{\psi}{\xi}\Big)  \Big|\Big(\frac{\xi}{\psi}\Big)^{d-2}\delta \Big(\frac{\xi}{\psi}\Big)^{d} \!  \Big(\delta \Big(\frac{\xi}{\psi}\Big)^d-1\Big) \Big|  \leq \Big(    \frac{\psi}{\xi}\Big) \tau^{d-2}
\end{aligned}
\end{equation*}
which together with \eqref{GHPROP1}, \eqref{GHPROP2}, \eqref{DDEF}, \eqref{PSIPROP}$_2$, and \eqref{DGPRODBND} gives the estimate
\begin{equation}\label{DEST}
\begin{aligned}
\big|D(\xi \SP; \varphi_0,v_0)\big| & \, \leq  \, \Big| \Big(\frac{\xi}{\psi}\Big)^{2d-3}
  \delta \Big(\delta \Big(\frac{\xi}{\psi} \Big)^d -1 \Big)\Big|\big(\xi^2 \varphi_0^2 + g''(0)\big) \\
& \,\quad + \Big|\Big(\frac{\xi}{\psi}\Big)^{d-2} g'(\delta\big(\tfrac{\xi}{\psi}\big)^{d-1})\Big| + \Big|\Big(\frac{\xi}{\psi}\Big)^{d-2}g'(\tfrac{\psi}{\xi})\Big|\\
&\leq \Big( \frac{\psi}{\xi} \Big) (1+\tau^{d}) (\varphi_0^2 + g''(0)) + 2
\widehat{\gamma}\SP(1+\tau^{d-2})\,, \;\; \xi\in(0,\tau]\,.
\end{aligned}
\end{equation} }

%

Next, since $\varphi_0 \in (0,\eps_{\tau}) $, by \eqref{PHI11UPOS} we have
\begin{equation}\label{QLBND}
 -\widehat{Q}(\xi \SP; \varphi_0,v_0) =
    \Phi_{11}({\psi'},\tfrac{\psi}{\xi})-\varphi_0^2 \xi^2 \SP
    > \SP \frac{\nu^2}{2}\,, \quad \xi \in (0,\tau]
\end{equation}
{and therefore, employing \eqref{EPSDEF}, \eqref{DEST} and \eqref{PSIPROP}, we obtain
\begin{equation}\label{I2EST}
\begin{aligned}
|I_2|  = \bigg| \Big(\frac{\xi}{\psi}\Big)\frac{D(\xi)}{\widehat{Q}(\xi)}\SP \xi\varphi_0^2 \bigg| & \leq \frac{2}{\nu^2} \bigg( (1+\tau^{d}) (\varphi_0^2 + g''(0)) + 2
\widehat{\gamma}\SP(1+\tau^{d-2})\Big( \frac{\xi}{\psi} \Big) \bigg)\xi \varphi_0^2\\
& \leq \frac{2 \varphi_0}{\nu^2} \Big( (\tau+\tau^{d+1}) (\varphi_0^2 + g''(0)) + 2
\widehat{\gamma}\SP(\tau^2+\tau^{d})\Big) \varphi_0\\
& \leq 4(1+\tau^{d+1}) (\varphi_0^2 + g''(0)+
\widehat{\gamma})\varphi_0\, \,, \quad \xi \in (0,\tau]\,,
\end{aligned}
\end{equation}
where we used the fact that $\frac{2\varphi_0}{\nu^2}<1$ because $0<\varphi_0<\eps_{\tau}$\,.  Thus by \eqref{DGPRODBND} and \eqref{I2EST} }we conclude
\begin{equation}\label{I23EST}
\begin{aligned}
|I_2|+|I_3| \leq  4  (1+\tau^{d+1}) (\varphi_0^2+g''(0)+ \widehat{\gamma})
 \varphi_0+2 \widehat{\gamma}(1+\tau^{d-2})\,, \,\, \xi \in (0,\tau].
\end{aligned}
\end{equation}
Combining \eqref{COMPTR2}, \eqref{I1EST} and \eqref{I23EST} we obtain
\eqref{RESCRADCBND}.
\end{proof}

\begin{lemma}
\label{BNDCRL}
 Assume $d\geq 2$ and $\Phi$ satisfies \eqref{spseis}-\eqref{D2HGROWTH}.
Let $(\psi,\delta)(\xi \SP; \varphi_0,v_0)$ be the
solution to \eqref{ODEIVPSYSRESC} defined on a maximal interval of existence $[0,\mathcal{T}(\varphi_0,v_0))$. Let $\tau>0$,
$v_M>v_m>0$ be fixed.  Then, for every
\begin{equation*}
0<  \varphi_0<\eps_{\tau},\quad v_m \leq v_0 \leq v_M
\end{equation*}
the interval $[0,\tau] \subset [0,\mathcal{T}(\varphi_0,v_0))$ and \begin{equation*}
\vspace{2pt}
\begin{aligned}
0 < c_1 &< \delta(\xi \SP; \varphi_0,v_0) <c_2, \quad  1 \leq \psi(\xi \SP; \varphi_0,v_0)
<c_3,\quad \forall\xi\in[0,\tau]
\end{aligned}
\end{equation*}
for some $c_1=c_1(\tau,v_m),\,c_2=c_2(\tau,v_M), \, c_3=c_3(\tau,v_M)$ independent of \SP$\varphi_0,v_0$.
\end{lemma}

\begin{proof}
Fix $\tau>0$, $v_M>v_m>0$. Take any $\varphi_0\in(0,\eps_{\tau})$, $v_0 \in [v_m,v_M]$ and consider the solution
$(\psi,\delta)(\xi \SP; \varphi_0,v_0)$ to \eqref{ODEIVPSYSRESC}. From \eqref{RESCRADC}
it follows that
\begin{equation*}
    \widehat{T}_{rad}(0; \varphi_0, v_0) := \lim_{\xi\to 0_+} \widehat{T}_{rad}(\xi \SP; \varphi_0, v_0) = h'(v_0)
\end{equation*}
and hence, using \eqref{RESCRADCBND}, we obtain for $\xi\in[0,\tau]$
\begin{equation*}
\begin{aligned}
    \big|\widehat{T}_{rad}&(\xi \SP;\varphi_0, v_0)-\widehat{T}_{rad}(0 \SP;\varphi_0, v_0)\big|\\
     &= \, \big|h'(\delta)+g'(\delta(\tfrac{\xi}{\psi})^{d-1})\Big( \frac{\xi}{\psi} \Big)^{d-1} - h'(v_0)\big|
  <  c_{rad}(1+\tau^{d+1})\tau\,.
\end{aligned}
\end{equation*}
{Then, similarly to \eqref{DGPRODBND}, using \eqref{GAMMAHDEF} and \eqref{PSIPROP}, we get
\begin{equation*}
\begin{aligned}
\Big|\Big(\frac{\xi}{\psi}\Big)^{d-1} g'(\delta(\tfrac{\xi}{\psi})^{d-1})\Big| \, \leq \, \widehat{\gamma} \SP (1+\tau^{d-2})\tau\,, \quad \xi\in(0,\tau]
\end{aligned}
\end{equation*}
and therefore
}
\begin{equation*}
    |h'(\delta(\xi \SP; \varphi_0,v_0))-h'(v_0)| <
    2(c_{rad}+\widehat{\gamma})(1+\tau^{d+2}),\quad \xi \in [0,\tau].
\end{equation*}
By \eqref{GHPROP1}, \eqref{DHATINF},  the inverse $h'^{-1}(z)$ is
well-defined for all $z\in(-\infty,\infty)$ and strictly positive. Hence
\begin{equation*}
 0 < c_1 < \delta(\xi \SP; \varphi_0, v_0) < c_2, \quad \xi \in [0,\tau]
\end{equation*}
with
\begin{equation*}
\begin{aligned}
    c_1(\tau, v_m) &:= h'^{-1}(h'(v_m)-2(c_{rad}+\widehat{\gamma})(1+ \tau^{d+2}))\\
     c_2(\tau, v_M)&:= h'^{-1}(h'(v_M)+2(c_{rad}+\widehat{\gamma})(1+ \tau^{d+2}))\,.
\end{aligned}
\end{equation*}

Next, by \eqref{ODEIVPSYSRESC}$_1$, \eqref{PSIPROP}$_1$ and the bound for $\delta$ we obtain
\begin{equation*}
    0< {\psi'}(\xi \SP; \varphi_0, v_0) = \delta \Big( \frac{\xi}{\psi}\Big)^{d-1} < \delta \tau ^ {d-1} <
    c_2
    \tau^{d-1}, \quad \xi \in (0,\tau]
\end{equation*}
and hence, since $\psi(0)=1$, we conclude
\begin{equation*}
 1 \leq \psi(\xi \SP; \varphi_0, v_0) < c_3(\tau,v_M)     :=1+ c_2 \tau^d \,, \quad \xi\in[0,\tau]\,.
\end{equation*}
\end{proof}

\subsection{The limiting system: $\varphi_0=0$}
In this section we consider the limiting problem $\varphi_0=0$ in \eqref{ODEIVPSYSRESC} and discuss its
connection to the problem of radial solutions for equilibrium elasticity analyzed by \cite{Ball82}.

\subsubsection{Inner solution}

Setting $\varphi_0=0$ into \eqref{ODEIVPSYSRESC} leads to
the initial value problem
\begin{equation}\label{ODEIVPSYSLIM}
\left\{
\begin{aligned}
{\psi_0'} &= \delta_0\Big(\frac{\xi}{\psi_0}\Big)^{d-1}\\
{\delta_0'} & = \bigg(\frac{d-1}{\psi_0}\bigg) \frac{ -\big(
\frac{\xi}{\psi_0}\big)^{2d-3}\delta_0 \Big( \delta_0 \big(\frac{\xi}{\psi_0}\big)^{d} -
1 \Big) g''( \delta_0 (\tfrac{\xi}{\psi_0})^{d-1})
\Big]}{\Bigl[-h''(\delta_0)-g''(\delta_0(\tfrac{\xi}{\psi_0})^{d-1})\big(\frac{\xi}{\psi_0}\big)^{2d-2}\Bigr]
}\\[2pt]
&\qquad +\bigg(\frac{d-1}{\psi_0}\bigg)\frac{ \big( \frac{\xi}{\psi_0}\big)^{d-2}
\Big[g'( \delta_0 (\tfrac{\xi}{\psi_0})^{d-1})-g'(\frac{\psi_0}{\xi})
\Big]}{\Bigl[-h''(\delta_0)-g''(\delta_0(\tfrac{\xi}{\psi_0})^{d-1})
\big(\frac{\xi}{\psi_0}\big)^{2d-2}\Bigr]\SP}\\
&\begin{aligned}
\psi_0(0 \SP; v_0)&=1\\
\delta_0(0\SP; v_0)&=v_0
\end{aligned}
\end{aligned}\right.
\end{equation}
while setting $\varphi_0=0$ into \eqref{RESCODE} gives the second order differential equation
\begin{equation}\label{LIMODE}
\begin{aligned}
-\Phi_{11}({\psi_0'},\tfrac{\psi_0}{\xi})\, {\psi_0''} =
\frac{d-1}{\xi}\Big({\psi_0'} -
\frac{\psi_0}{\xi}\Big)P({\psi_0'},\tfrac{\psi_0}{\xi})\,.
\end{aligned}
\end{equation}

First, we show that  \eqref{ODEIVPSYSLIM} has a globally-defined solution.
This is an {\it inner solution} that describes the behavior of a material with a cavity whose surface moves with infinitely small speed.
\begin{proposition}\label{LIMSOLLMM}
Let $d\geq 2$, let $\Phi$ satisfy \eqref{spseis}-\eqref{GGROWTH}, and  $v_0>0$ be fixed.
Then, there exists a unique global solution
\begin{equation*}
\psi_0(\xi \SP; v_0) \in C^2[0,\infty),\quad \delta_0(\xi \SP; v_0) \in C^1[0,\infty)
\end{equation*}
to the initial value problem \eqref{ODEIVPSYSLIM}
with $\psi_0$ solving the equation \eqref{LIMODE} on $(0,\infty)$.
The solution has the following properties:
\begin{itemize}
\item [$(i)$] \; $\psi_0'$, $\tfrac{\psi_0}{\xi}$, $(\psi_0'
-\tfrac{\psi_0}{\xi})$ are strictly monotonic and satisfy
\begin{align}
\label{MONTL1}
&{\psi_0''}(\xi)>0, \quad \frac{d}{d\xi}\Big(\frac{\psi_0}{\xi}\Big) < 0, \quad
\frac{d}{d\xi}\Big({\psi_0'} - \tfrac{\psi_0}{\xi}\Big) > 0, \quad
\xi\in(0,\infty)\,.
\\
\label{MONTL2}
&0 <  {\psi_0'}(\xi)  < {\psi_0'}(\tau)<  \frac{\psi_0(\tau)}{\tau} < \frac{\psi_0(\xi)}{\xi}\, ,
\quad \mbox{$0<\xi<\tau<\infty$}.
\end{align}

\item[$(ii)$] \; The following limits exist and satisfy
\begin{align}
&0 \, < \, \lim_{\xi\to \infty} {\psi_0'} = \lim_{\xi\to \infty} \Big(
\frac{\psi_0}{\xi}
\Big)=:\Lambda_0(v_0) \, < \, \infty \quad \mbox{for some} \quad \Lambda_0=\Lambda_0(v_0)>0\,. \label{ABLIML}
\end{align}

\item[$(iii)$] $\psi_0,\delta_0$ obey the bounds
\begin{align}
\label{LIMSOLBND1}
\max(1,\Lambda_0\xi) < \psi_0(\xi) & < 1+\Lambda_0 \xi, \quad 0 <
{\psi_0'}(\xi) < \Lambda_0
\\
\label{LIMSOLBND2}
0&< \mu_1 < \delta_0(\xi) <\mu_2\,,
\end{align}
for $\xi\in[0,\infty)$  and some $\mu_1 = \mu_1(v_0)$, $\mu_2=\mu_2(v_0)$. In addition,
\begin{equation}\label{DDELTAINTEGR}
{\delta_0'}(\xi) \in L^1(0,\infty)\,.
\end{equation}
\end{itemize}

\end{proposition}

\begin{proof}
The first part of the proof proceeds along the lines of the proof of Theorem \ref{CAVSOLEXISTSP}.
Fix $v_0>0$. Using \eqref{GGROWTH},   there exists a unique local solution   $(\psi_0,\delta_0)$ of
\eqref{ODEIVPSYSLIM}. Clearly, $\psi_0$ satisfies \eqref{LIMODE}, and if we set
$$
a_0 = \psi'_0 \, , \quad b_0 = \frac{\psi_0}{\xi} \, ,
$$
$(a_0, b_0)$ also satisfies the system of ordinary differential equations
\begin{equation}
\label{sysab}
\begin{aligned}
a_0' &=  - \frac{d-1}{\xi} (a_0 - b_0 ) \frac{P(a_0, b_0)}{\Phi_{11}(a_0,b_0)}
\\
b_0' &= \frac{1}{\xi} (a_0 - b_0)
\end{aligned}
\end{equation}
By construction, the solution $(a_0, b_0)$ starts above the diagonal line of equilibria $a_0 = b_0$ for \eqref{sysab}.
Using \eqref{DPHISIGN1} we see that the monotonicity properties \eqref{MONTL1} and \eqref{MONTL2} hold, that
the solution $(a_0, b_0)$ is globally defined, and that
the following  limits exist
\begin{equation*}
\lim_{\xi \to \infty} {\psi_0'} = A \, \leq \, \lim_{\xi \to \infty}
\Big(\frac{\psi_0}{\xi}\Big) = B,\quad \lim_{\xi\to \infty} \delta_0 = AB^{d-1}
\end{equation*}
for some finite $A,B>0$.

From \eqref{sysab} we deduce that
$$
\frac{d}{d\xi} (b_0 - a_0 ) + \frac{1}{\xi} \Big ( 1 +  (d-1) \frac{P(a_0, b_0)}{\Phi_{11}(a_0,b_0)} \Big ) (b_0 - a_0) = 0
$$
and using again \eqref{DPHISIGN1}
\begin{equation}
\label{derivbou}
\begin{aligned}
0 < (b_0 - a_0 )(\xi)  &=  (b_0 - a_0 ) (1) \exp \Big  \{  - \int_1^\xi \frac{1}{\zeta} \Big ( 1 +  (d-1) \frac{P(a_0, b_0)}{\Phi_{11}(a_0,b_0)} \Big ) d\zeta \Big \}
\\
&\le \frac{ (b_0 - a_0 ) (1)}{\xi}
\end{aligned}
\end{equation}
Hence, $A=B =: \Lambda_0$, and the monotonicity properties \eqref{MONTL1} yield the bounds
\begin{equation*}
1 < \psi_0\,, \quad
0 < \psi'_0  < \Lambda_0\,,\quad
\Lambda_0 \xi < \psi_0 < 1 + \Lambda_0 \xi
\end{equation*}
thus providing \eqref{LIMSOLBND1} and \eqref{LIMSOLBND2}.
Finally, recalling \eqref{ODEIVPSYSLIM}$_1$, we have
\begin{equation*}
    {\delta_0'}
    ={\psi_0''}\Big(\frac{\psi_0}{\xi}\Big)^{d-1}+(d-1){\psi_0'}\Big(\frac{\psi_0}{\xi}\Big)^{d-2}\frac{d}{d\xi}\Big(\frac{\psi_0}{\xi}\Big)\,.
\end{equation*}
Moroever, \eqref{ABLIML}, \eqref{derivbou} and \eqref{LIMODE} imply
$$
\left | \frac{d}{d\xi}\Big(\frac{\psi_0}{\xi} \Big ) \right |  \le  \frac{C}{\xi^2} \, , \quad 0 < \psi_0'' \le  \frac{C}{\xi^2} \, , \quad
\mbox{for  $\xi \ge 1$},
$$
and thus $\delta_0'$ is integrable on $[1,\infty)$. Since
${\delta}_0\in C^1[0,\infty)$, we conclude with \eqref{DDELTAINTEGR}.
\end{proof}

\par\smallskip

For future purposes we introduce the notation
\begin{align}
f_0(\xi \SP; v_0) &=
\Phi_{11}({\psi_0'},\tfrac{\psi_0}{\xi})\Big(\frac{\xi}{\psi_0}\Big)^{2d-2} \notag\\
&=h''(\delta_0)+g''(\delta_0(\tfrac{\xi}{\psi_0})^{d-1})
\Big(\frac{\xi}{\psi_0}\Big)^{2d-2}\label{FZDEF}\\
D_0(\xi \SP; v_0) &= -\Big( \frac{\xi}{\psi_0}\Big)^{2d-3}\delta_0 \Big( \delta_0
\Big(\frac{\xi}{\psi_0}\Big)^{d} - 1 \Big) g''( \delta_0 (\tfrac{\xi}{\psi_0})^{d-1}) \notag\\
& \;\;\quad + \Big( \frac{\xi}{\psi_0}\Big)^{d-2} \big[g'( \delta_0
(\tfrac{\xi}{\psi_0})^{d-1})-g'(\tfrac{\psi_0}{\xi}) \big] \,.
\label{DZDEF}
\end{align}


\subsubsection{Connections to equilibrium elasticity. Representations of $\Lambda_0$}\label{STATCONNECT}

In \cite{Ball82} J. Ball considers the boundary-value problem for the equations of equilibrium elasticity
\begin{equation}\label{STATPDE}
\left\{
\begin{aligned}
\mathrm{div} \SP \bigg(\frac{\del W}{\del F}(\nabla{y}) \bigg) &= 0\,, \quad \quad x\in \mathcal{B}=\{x\in\RR^d: |x|<1\} \\[2pt]
y(x) &= \lambda x\,, \quad \quad x\in\del \mathcal{B}
\end{aligned}\right.
\end{equation}
for some stretch $\lambda>0$ and studies radial solutions to \eqref{STATPDE} that have the form
\begin{equation}\label{RADSOL}
  y(x) = \frac{w(|x|)}{|x|}\, \quad \mbox{with} \quad w(R):[0,\infty) \to [0,\infty)\,.
\end{equation}
The amplitude $w$ of such solution satisfies the boundary-value problem
\begin{equation}\label{STATRPDE}
\left\{
\begin{aligned}
\frac{d}{dR} \Big(R^{d-1} \Phi_1\big(w',\frac{w}{R}\big) \Big) - (d-1)R^{d-2} \Phi_2\big(w',\frac{w}{R}\big)&=0\\
w(1) & = \lambda\,.
\end{aligned}
\right.
\end{equation}

\par\smallskip

It is proved in \cite[Theorem 7.9]{Ball82} that a solution $w_{\lambda}$ of \eqref{STATRPDE} with $w_{\lambda}(0)>0$ and $T_{rad}[w_{\lambda}](0)=0$ exists if and only if $\lambda > \lambda_{cr}$ where
\begin{equation}\label{STATSTRETCH}
  \lambda_{cr}:= \lim_{R\to \infty} \frac{w(R)}{R}
\end{equation}
and $w(R)$ is any solution to radial elastostatics \eqref{STATRPDE}$_1$ with $w(0)>0$ and $T_{rad}[w](0)=0$ (such solutions exist by \cite[Theorems 7.3, 7.7]{Ball82}); here
\begin{equation*}
  T_{rad}[w](R) := \Phi_1\big(w',\frac{w}{R}\big)\Big(\frac{R}{w}\Big)^{d-1}
\end{equation*}
denotes the radial component  of the Cauchy stress. Moreover, the solution $w_{\lambda}$ generates a cavitating weak solution to \eqref{STATPDE} via the formula \eqref{RADSOL}.

\par\smallskip

Let $(\psi_0,\delta_0)(\xi\SP; v_0)$ be the solution of \eqref{ODEIVPSYSLIM}. Then, recalling \eqref{PDEF}, observe that \eqref{LIMODE} translates into \begin{equation*}
\begin{aligned}
\frac{d}{d\xi} \big(\xi^{d-1} \Phi_1\big(\psi_0',\frac{\psi_0}{\xi}\big) \big) - (d-1)\xi^{d-2} \Phi_2\big(\psi_0',\frac{\psi_0}{\xi}\big)=0\,.
\end{aligned}
\end{equation*}
Thus, $\psi_0(\xi_0 \SP ; v_0)$ is a solution to \eqref{STATRPDE}, with $\lambda=\psi_0(1 \SP; v_0)$.  By \eqref{CSTRESSBND1}
\begin{equation*}
  T_{rad}[\psi_0( \cdot \SP;  v_0)](0) = 0 \quad \mbox{if and only if}\quad v_0=H\,
\end{equation*}
and hence  \eqref{ABLIML}, \eqref{STATSTRETCH} imply that the {\em critical stretch} $\lambda_{cr}$  for radial equilibrium elasticity
 (in the stress-free case) satisfies
\begin{equation}\label{CONNECT}
  \lambda_{cr} = \Lambda_0(H)\,.
\end{equation}


\par\medskip

\noindent{\bf Representations of $\Lambda_0$.} We briefly discuss some of the representations of $\Lambda_0(v_0)$ defined in \eqref{ABLIML}.
Suppose $d\geq 2$, $\Phi$ satisfies \eqref{spseis}-\eqref{GGROWTH} and $(\psi_0,\delta_0)(\xi\SP; v_0)$ is the solution
of \eqref{ODEIVPSYSLIM}. Then \eqref{ODEIVPSYSLIM}$_1$ and \eqref{ABLIML} imply
\begin{equation*}
 \delta_0(\xi \SP ; v_0) = \Big[ {\psi_0'}\Big(\frac{\psi_0}{\xi}\Big)^{d-1}\Big](\xi\SP ; v_0)
  \rightarrow (\Lambda_0(v_0))^d \quad
 \mbox{as} \quad \xi\to \infty\,.
\end{equation*}
Thus, integrating \eqref{ODEIVPSYSLIM}$_2$ and recalling \eqref{DDELTAINTEGR}, we conclude
\begin{equation}\label{REPRFORM1}
    \Lambda_0(v_0) = \sqrt[d]{v_0+\int^{\infty}_0 (1-d)\bigg[\frac{D_0}{\psi_0 f_0}\bigg](\xi\SP; v_0) \, d\xi }\,,
\end{equation}
with $f_0$, $D_0$ defined by \eqref{FZDEF}, \eqref{DZDEF}. Since the limiting system \eqref{ODEIVPSYSLIM} is non-singular, the integral in \eqref{REPRFORM1} is well-defined; its rate of convergence is $O(\xi^{-1})$.

Another representation can be obtained by differentiating
\begin{equation*}
\begin{aligned}
    T_{rad}[\psi_0( \xi \SP;  v_0)] &= \Big[ h'(\delta_0)+g'(\delta_0(\tfrac{\xi}{\psi_0})^{d-1})\Big( \frac{\xi}{\psi_0} \Big)^{d-1}\Big](\xi; v_0)\,
\end{aligned}
\end{equation*}
which, in view of \eqref{ODEIVPSYSLIM}$_2$, gives
\begin{equation}\label{ODET0}
\begin{aligned}
\frac{d}{d\xi}  & \Big( T_{rad}[\psi_0( \xi \SP;  v_0)] \Big) =
 \Big\{ \frac{d-1}{ \psi_0}\Big(\frac{\xi}{\psi_0}\Big)^{ d-2 }
 \Big[  g'(\tfrac{\psi_0}{\xi}) -
 \delta\Big(\frac{\xi}{\psi_0}\Big)^{d}g'(\delta_0\big(\tfrac{\xi}{\psi_0}\big)^{d-1})\Big]\Big\} (\xi \SP ; v_0)\,.
\end{aligned}
\end{equation}
(We note that \eqref{ODET0} can be directly obtained by setting $\varphi_0 = 0$ in \eqref{COMPTR2}).
Integrating the above identity over $(0,\xi)$, letting $\xi \to \infty$ and using \eqref{ABLIML} leads to
\begin{equation*}
\begin{aligned}
\chi(\Lambda_0(v_0)) =
 h'(v_0) + \int_0^{\infty} \!\!\biggl \{ \frac{d-1}{\psi_0}\Big(\frac{\xi}{\psi_0}\Big)^{d-2} \,
 \Big[ g'(\tfrac{\psi_0}{\xi}) -
 \delta_0\Big(\frac{\xi}{\psi_0}\Big)^{d}g'(\delta_0\big(\tfrac{\xi}{\psi_0}\big)^{d-1})\Big] \bigg\}(\xi\SP; v_0) \, d\xi
\end{aligned}
\end{equation*}
where
\begin{equation*}
  \chi(x) :=  \Phi_{1}  (x,x)x^{1-d}= h'(x^d)+g'(x)x^{1-d}, \quad x\in(0,\infty)\,.
\end{equation*}

Suppose that
\begin{equation}\label{chiMONT}\tag{H6}
\chi'(x)>0 \quad \mbox{and} \quad \limsup_{x \to 0_+} \big(h'(x)x\big) < 0.
\end{equation}
In that case, the inverse $\chi^{-1}$ is  well-defined and hence
\begin{equation}\label{REPRFORM2}
\begin{aligned}
 &\Lambda_0(v_0)= \\
 & \quad \chi^{-1}\bigg(h'(v_0) + \int_0^{\infty} \bigg[ \frac{d-1}{\xi}\Big(\frac{\xi}{\psi_0}\Big)^{d-2}\!\!
 \Big( g'(\tfrac{\psi_0}{\xi}) -
 \delta_0\Big(\frac{\xi}{\psi_0}\Big)^{d}g'(\delta_0\big(\tfrac{\xi}{\psi_0}\big)^{d-1})\bigg](\xi;v_0) \SP d\xi\bigg)\,.
\end{aligned}
\end{equation}
For further ideas on the representation of $\Lambda_0(v_0)$ we refer the reader to \cite{Ball82}.

{From the representation formulas \eqref{REPRFORM1}, \eqref{REPRFORM2} one can derive lower bounds for the critical stretch for cavitation
$\Lambda_0 (v_0)$. We consider first the formula \eqref{REPRFORM2} and impose in addition to \eqref{chiMONT} the
hypothesis
\begin{equation}\label{BE}\tag{H7}
 \big( g'(x)x \big)' \, \geq 0 \, , \quad x\in(0,\infty)  \, .
\end{equation}
The latter is the expression of the Baker-Ericksen inequality for stored energies of class \eqref{spseis} and is implied by rank-one convexity
(see Proposition \ref{appball}).
Using \eqref{REPRFORM2}  in conjunction to \eqref{ODEIVPSYSLIM}$_1$, \eqref{MONTL2},
\eqref{chiMONT} and \eqref{BE} we obtain
\begin{equation}
\label{critstretch1}
\Lambda_0(v_0) \,> \, \chi^{-1}(h'(v_0)) \,.
\end{equation}
}

Next we consider the representation formula \eqref{REPRFORM1} and impose the hypothesis
\begin{equation}\label{DDGX}\tag{H8}
 \big( g''(x)x \big)' \, \geq 0 \, , \quad x\in(0,\infty)  \, .
\end{equation}
Recall that $D_0$ is defined in \eqref{DZDEF} and set $a=\dot{\psi}_0$, $b=\frac{\psi_0}{\xi}$.
Using \eqref{DDGX}, and the facts $a<b$ and $g''' \leq 0$, we obtain
\begin{align*}
  D_0 
      &=  - b^{2-d} a ( ab^{-1} -1 ) g''(a) + b^{2-d}(g'(a)-g'(b))\\
      &=  b^{1-d}(a-b) \Big( -a g''(a) + b\frac{g'(a)-g'(b)}{a-b} \Big)\\
      & \, \leq \, b^{1-d}(b-a) (ag''(a) - b g''(b))  \\
      & \, \leq \, 0 \, .
 \end{align*}
In turn,   \eqref{REPRFORM1}, \eqref{FZDEF} and \eqref{GHPROP1}  imply the lower bound
\begin{equation}
\label{critstretch2}
\Lambda_0(v_0) > \sqrt[d]{v_0}\, .
\end{equation}

\subsection{Convergence to the limiting solution} \label{secCONVLIMSOL}

In this section we study the convergence of solutions  $(\psi,\delta) = (\psi,\delta)(\xi \SP; \varphi_0, V(\varphi_0))$
of the system \eqref{ODEIVPSYSRESC}  to solutions $(\psi_0,\delta_0) = (\psi_0,\delta_0)(\xi \SP; V(0))$ of the system
\eqref{ODEIVPSYSLIM}.
Using the notations \eqref{FDEF}-\eqref{DDEF}, \eqref{FZDEF}-\eqref{DZDEF}, we rewrite the equations \eqref{ODEIVPSYSRESC}$_2$,
\eqref{ODEIVPSYSLIM}$_2$ as follows
\begin{equation}\label{DDELTAODESHORT}
\begin{aligned}
{\delta'} &= -\bigg[\frac{d-1}{\psi} \Big( 1 - \frac{\varphi_0^2
\xi^2}{\widehat{Q}}\Big)\frac{D}{f}\bigg](\xi; \varphi_0, V(\varphi_0))  \\
{\delta_0'} &= -\bigg[\frac{d-1}{\psi_0}
\frac{D_0}{f_0}\bigg](\xi \SP; V(0))\,.
\end{aligned}
\end{equation}
Recall also that the initial data satisfy
\begin{equation}
\label{compdata}
\begin{aligned}
 \psi(0 \SP; \varphi_0,V(\varphi_0)) -  \psi_0(0 \SP; V(0)) &= 0  \, , \quad
 \\
 \delta(0 \SP; \varphi_0,V(\varphi_0)) -  \delta_0(0 \SP; V(0))  &= V(\varphi_0)-V(0) \, .
\end{aligned}
\end{equation}

\begin{proposition}
\label{proprescl}
Assume $d\geq 2$, $\Phi$ satisfies \eqref{spseis}-\eqref{D2HGROWTH}, and {$V:[0,\infty)\to (0,\infty)$ is a continuous function} defined by \eqref{CSTRESSBNDV1}. Let $(\psi,\delta)$, $(\psi_0,\delta_0)$ be the solutions of \eqref{ODEIVPSYSRESC}, \eqref{ODEIVPSYSLIM}, respectively. Let $\tau>0$ be fixed. Then, there exists $\widehat{\eps}_{\tau}>0$ such that for every
\begin{equation*}
0<\varphi_0 < \widehat{\eps}_{\tau}
\end{equation*}
 the interval $[0,\tau]\subset[0,\mathcal{T}(\varphi_0,V(\varphi_0)))$ and
\begin{equation}\label{ENERGYGROWTH}
\begin{aligned}
\Big( \psi(\xi \SP; \varphi_0 , V(\varphi_0)) - \psi_0(\xi \SP; V(0))\Big)^2
&+ \Big(\delta(\xi \SP ; \varphi_0 , V(\varphi_0))-\delta_0(\xi \SP; V(0))\Big)^2
\\
& \, \leq
 \widehat{C}\SP\Big( \big(V(\varphi_0)-V(0)\big)^2 + \varphi_0^4\Big)\,, \qquad  \forall \xi\in[0,\tau]
\end{aligned}
\end{equation}
with the constant $\widehat{C}=\widehat{C}(\tau,V(0))>0$ and independent of $\varphi_0$. As a consequence,
\begin{equation}\label{CONVRG}
(\psi,\delta)(\xi\SP ; \varphi_0, V(\varphi_0))  \; \to \;  (\psi_0,\delta_0)(\xi\SP ; V(0))\quad \mbox{as} \quad \varphi_0  \to 0
\end{equation}
uniformly on $[0,\tau]$.
\end{proposition}

\begin{proof}
Fix $\tau>0$. In view of continuity of $V$, there exists $\tilde{\eps}>0$ such that
\begin{equation}\label{VBOUNDZ}
\frac{1}{2} V(0) < V(\varphi_0) < 2V(0)\,, \quad \forall \varphi_0 \in (0,\tilde{\eps})\,.
\end{equation}
Set $\widehat{\eps}_{\tau}:=\min(\eps_{\tau},\tilde{\eps})$, with $\eps_{\tau}$  defined in \eqref{EPSDEF}. Then, Lemma \ref{BNDCRL},
\eqref{LIMSOLBND1}$_1$, \eqref{LIMSOLBND2}, \eqref{VBOUNDZ}  imply
\begin{equation*}
[0,\tau]\subset [0,\mathcal{T}(\varphi_0,V(\varphi_0)))\, \quad \mbox{for every} \quad \varphi_0 \in (0,\widehat{\eps}_{\tau})
\end{equation*}
and there exist constants $m_1,m_2,m_3>0$ that depend on $\tau,V(0)$  and are independent of $\varphi_0$ such that
\begin{equation}\label{SOLBNDS}\vspace{2pt}
\begin{aligned}
\qquad 0 < \SP m_1 \SP & < \SP \delta(\xi \SP; \varphi_0,V(\varphi_0)), \, \delta_0(\xi \SP; V(0)) \SP < \SP m_2\\[3pt]
\qquad 1 \SP  & \leq \SP \psi(\xi \SP; \varphi_0,V(\varphi_0)),\, \psi_0(\xi \SP;V(0)) \SP
< \SP m_3
\end{aligned}
\end{equation}
for all $\varphi_0\in(0,\widehat{\eps}_{\tau})$, $\xi \in [0,\tau]$.

\par\medskip

Next, we fix $\varphi_0 \in (0,\widehat{\eps}_{\tau})$ and derive an energy identity that monitors the distance between the
solutions $(\psi,\delta)$ and $(\psi_0,\delta_0)$ of the systems \eqref{ODEIVPSYSRESC}, \eqref{ODEIVPSYSLIM}, respectively.
First, we subtract \eqref{ODEIVPSYSLIM}$_1$ from \eqref{ODEIVPSYSRESC}$_1$ to get the identity
\begin{equation}\label{DPSIDIFF}
\begin{aligned}
{\psi'} - {\psi_0'} &= \delta \Big( \frac{\xi}{\psi}\Big)^{d-1} - \delta_0 \Big(
\frac{\xi}{\psi_0}\Big)^{d-1}\\
&= (\delta - \delta_0) \Big(\frac{\xi}{\psi}\Big)^{d-1} +
(\psi_0-\psi)\frac{\delta_0}{\psi_0} \Big(\frac{\xi}{\psi}\Big)^{d-1} \sum_{i=0}^{d-2}
\Big( \frac{\psi}{\psi_0} \Big)^i\,.
\end{aligned}
\end{equation}
Subtracting the two equations in \eqref{DDELTAODESHORT} gives
\begin{equation}\label{DDELTADIFF}
\begin{aligned}
\frac{1}{d-1} (\delta' - \delta_0')  = \frac{(D_0-D)}{\psi_0 f_0}+\frac{(\psi -
\psi_0)}{\psi \psi_0 f_0}D  + \frac{(f-f_0)}{\psi f f_0}D  + \frac{ \SP \varphi_0^2\SP
\xi^2}{\widehat{Q} \SP\psi f}D\,.
\end{aligned}
\end{equation}
Multiplying \eqref{DPSIDIFF} by $(\psi-\psi_0)$ and \eqref{DDELTADIFF} by $(\delta-\delta_0) (d-1)$ and adding the results we obtain
\begin{equation}\label{ENERGYID}
\begin{aligned}
 \frac{d}{d \xi} & \Big( \frac{1}{2}(\psi - \psi_0)^2 + \frac{1}{2}(\delta-\delta_0)^2
\Big)  \\
&=(\psi-\psi_0)(\delta - \delta_0) \Big(\frac{\xi}{\psi}\Big)^{d-1} \! -
(\psi-\psi_0)^2\frac{\delta_0}{\psi_0} \Big(\frac{\xi}{\psi}\Big)^{d-1} \sum_{i=0}^{d-2}
\Big( \frac{\psi}{\psi_0} \Big)^i\\
&\qquad  (d-1)\bigg[- (D-D_0)(\delta-\delta_0)\frac{1}{\psi_0 f_0}
+(\psi -\psi_0)(\delta-\delta_0)\frac{D}{\psi \psi_0 f_0}\\
&\qquad\qquad\qquad  + (f-f_0)(\delta-\delta_0)\frac{D}{\psi f f_0}+ (\delta-\delta_0)\frac{D \SP
\varphi_0^2 \SP \xi^2}{\widehat{Q} \SP \psi f}\bigg]\,,\qquad \xi\in(0,\tau]\,.
\end{aligned}
\end{equation}

\par\smallskip

We now estimate the right-hand side of \eqref{ENERGYID}. By \eqref{SOLBNDS}
we get for $\xi \in (0,\tau]$
\begin{equation}\label{CFFBND1}
\begin{aligned}
0<\Big(\frac{\xi}{\psi}\Big)^{d-1} \!\leq \tau^{d-1}\,, \quad \Big|\frac{\delta_0}{\psi_0}
\Big(\frac{\xi}{\psi}\Big)^{d-1} \sum_{i=0}^{d-2} \Big(
\frac{\psi}{\psi_0} \Big)^i \Big| \, \leq \, \tau^{d-1}m_2 \sum_{i=0}^{d-2}m_3^i\,.
\end{aligned}
\end{equation}
Similarly, by \eqref{GHPROP2}, \eqref{QLBND}, \eqref{FDEF}, \eqref{FZDEF}, and \eqref{SOLBNDS}, we obtain


\begin{equation}\label{CFFBND2}
\begin{aligned}
\Big| \frac{1}{\psi_0 f_0}\Big| +\Big| \frac{1}{\psi \psi_0 f_0}\Big|+\Big|
\frac{1}{\psi
 ff_0}\Big| + \Big| \frac{\xi^2}{\widehat{Q}\SP\psi f}\Big| \, \leq  \, \SP\frac{\max(1,h''(m_2))}{(h''(m_2))^2} \Big( 3+\frac{2\tau^2}{\nu^2}\Big)\,.
\end{aligned}
\end{equation}
Also, using \eqref{GHPROP2}, \eqref{DDEF}, \eqref{SOLBNDS} and the inequalities \eqref{PSIPROP}$_2$, \eqref{DGPRODBND}, we obtain
\begin{equation}\label{DBND}
\begin{aligned}
\big|D(\xi \SP; \varphi_0, V(\varphi_0))\big| \, \leq \, \tau^{2d-3} m_2 (\tau^2 \varphi_0^2 +
g''(0)) + 2 \widehat{\gamma}\SP(1+\tau^{d-2})\,.
\end{aligned}
\end{equation}
Combining \eqref{ENERGYID}--\eqref{DBND}, and using Young's inequality, we get for
$\xi\in(0,\tau]$
\begin{equation}\label{ENERGYEST1}
\begin{aligned}
\frac{d}{d \xi} & \Big((\psi - \psi_0)^2 + (\delta-\delta_0)^2\Big)  \\
& \quad \leq \, C\Big((\psi-\psi_0)^2 + (\delta-\delta_0)^2 +(f-f_0)^2+(D-D_0)^2+
\varphi_0^4\Big)
\end{aligned}
\end{equation}
with $C=C(\tau,V(0))>0$ independent of $\varphi_0$.

\par\smallskip

By \eqref{SOLBNDS}, \eqref{DPSIDIFF}
and \eqref{CFFBND1} for all $\xi\in(0,\tau]$
\begin{equation}\label{FFZDIFF}
\begin{aligned}
\big|f(\xi &\SP;   \varphi_0,  V(\varphi_0))-f_0(\xi\SP; V(0))\big|\\
&\,\quad \leq \, |h''(\delta)-h''(\delta_0)|+ g''(\delta_0(\tfrac{\xi}{\psi_0})^{d-1})
\Big|\Big(\frac{\xi}{\psi}\Big)^{ 2d-2 }-\Big(\frac{\xi}{\psi_0}\Big)^{2d-2}\big|\\
& \;\;\qquad +\Big(\frac{\xi}{\psi}\Big)^{2d-2} |g''(\delta(\tfrac{\xi}{\psi})^{d-1})
-g''(\delta_0(\tfrac{\xi}{\psi_0})^{d-1})| \\
&\, \quad \leq \, |\delta-\delta_0| \SP
\sup_{x\in[m_1,\,m_2]}|h'''(x)| + \tau^{2d-2} |\psi_0-\psi|\,   g''(0) \sum_{i=0}^{2d-3} m_3^i\\
& \,\,\qquad+\tau^{3d-3} \Big(|\delta - \delta_0| + |\psi_0-\psi| m_2 \sum_{i=0}^{d-2} m_3
^i \Big) \sup_{x\in[0,m_2\tau^{d-1}]} |g'''(x)|\,.
\end{aligned}
\end{equation}

Next, by \eqref{DDEF} and \eqref{DZDEF}
\begin{equation}\label{DDZDIF}
\begin{aligned}
&|D(\xi \SP; \varphi_0, V(\varphi_0))  - D_0(\xi \SP; V(0))|\\[3pt]
&\quad  \leq \Bigg\{\,\Big|\Big(\frac{\xi}{\psi}\Big)^{2d-3}
  \delta \Big(\delta \Big(\frac{\xi}{\psi} \Big)^d -1 \Big)
\, \xi^2 \varphi_0^2 \Big|\\
& \;\quad\qquad + \Big| \Big(\frac{\xi}{\psi_0}\Big)^{2d-3} -
\Big(\frac{\xi}{\psi}\Big)^{2d-3}\Big| \Big| \delta_0 \Big(\delta_0
\Big(\frac{\xi}{\psi_0} \Big)^d -1
\Big) g''(\delta_0\big(\tfrac{\xi}{\psi_0}\big)^{d-1}) \Big|\\
& \;\quad\qquad + \Big| \delta_0  - \delta \Big| \Big|
\Big(\frac{\xi}{\psi}\Big)^{2d-3}\Big(\delta_0 \Big(\frac{\xi}{\psi_0} \Big)^d -1
\Big) g''(\delta_0\big(\tfrac{\xi}{\psi_0}\big)^{d-1}) \Big|\\
& \;\qquad\quad + \Big| \delta_0 \Big(\frac{\xi}{\psi_0} \Big)^d  - \delta
\Big(\frac{\xi}{\psi} \Big)^d  \Big| \Big| \delta
\,\Big(\frac{\xi}{\psi}\Big)^{2d-3}g''(\delta_0\big(\tfrac{\xi}{\psi_0}\big)^{d-1})
\Big|\\
& \;\qquad\quad + \Big| g''(\delta_0\big(\tfrac{\xi}{\psi_0}\big)^{d-1})  -
g''(\delta\big(\tfrac{\xi}{\psi}\big)^{d-1}) \Big| \Big|  \Big(
\delta\Big(\frac{\xi}{\psi} \Big)^d - 1\Big)\delta
\,\Big(\frac{\xi}{\psi}\Big)^{2d-3}\Big|\Bigg\}\\
& \;\qquad\quad + \Big| \Big(\frac{\xi}{\psi}\Big)^{d-2}
g'(\delta\big(\tfrac{\xi}{\psi}\big)^{d-1})-\Big(\frac{\xi}{\psi_0}\Big)^{d-2}
g'(\delta_0\big(\tfrac{\xi}{\psi_0}\big)^{d-1})\Big|\\
& \;\qquad\quad + \Big| \Big(\frac{\xi}{\psi}\Big)^{d-2} g'(\tfrac{\psi}{\xi}) -
\Big(\frac{\xi}{\psi_0}\Big)^{d-2} g'(\tfrac{\psi_0}{\xi}) \Big|=: J_1+J_2+J_3\,.
\end{aligned}
\end{equation}
Using \eqref{GHPROP2}, \eqref{PSIPROP}$_2$, \eqref{SOLBNDS}, \eqref{DPSIDIFF}, and \eqref{CFFBND1} we obtain
\begin{equation}\label{J1J5EST}
\begin{aligned}
J_1 & \, \leq   \,\tau^{2d-1} m_2 \SP \varphi_0^2 + \tau^{2d-3}|\psi_0-\psi| g''(0) \SP m_2 \!
\sum_{i=0}^{2d-4}m_3^i + \tau^{2d-3} g''(0) |\delta_0-\delta|\\
& \;\quad + \,m_2 \,\tau^{3d-3}g''(0)\,\Big(|\delta -
\delta_0| +|\psi_0-\psi| m_2 \sum_{i=0}^{d-1}m_3^i\Big)  \\
& \;\quad + m_2\tau^{3d-4}\Big(|\delta - \delta_0| + |\psi_0-\psi| m_2
\sum_{i=0}^{d-2}m_3^i\Big) \sup_{x\in[0,m_2\tau^{d-1}]} |g'''(x)|.
\end{aligned}
\end{equation}

To estimate terms $J_2$, $J_3$ we consider two separate cases for the constant $d$.

\par\smallskip

\noindent{\it Case $ d=2$.} First, using \eqref{GHPROP2},
 \eqref{DPSIDIFF}, \eqref{CFFBND1}, we obtain
\begin{equation}\label{J6EST1}
\begin{aligned}
J_2 & 
  \leq \, g''(0) \,\tau^{d-1} \Big(|\delta - \delta_0| + |\psi_0-\psi|\,m_2
\sum_{i=0}^{d-2}m_3^i \Big)\,.
\end{aligned}
\end{equation}
Then by
\eqref{GHPROP1}, \eqref{GHPROP2}, and \eqref{GGROWTH} we obtain for $x\in[0,\infty)$
\begin{equation}\label{D2GXBND}
  0 \SP \leq \SP g''(x)x \SP \leq \SP \int_0^1 \, g''(sx)x\, ds = g'(x)-g'(0)\, \leq \gamma-g'(0)\,.
\end{equation}
Then, using \eqref{GHPROP2}, \eqref{D2GXBND}, and the fact that $\psi,\psi_0 \geq
1$, we obtain
\begin{equation}\label{J7EST1}
\begin{aligned}
J_3  \, = \, \Big| g'(\tfrac{\psi}{\xi})-g'(\tfrac{\psi_0}{\xi})\Big| \, \leq \,
g''(\tfrac{1}{\xi})\Big|\frac{\psi}{\xi} - \frac{\psi_0}{\xi}\Big| \, \leq \,
(\gamma-g'(0))\,|\psi - \psi_0|.
\end{aligned}
\end{equation}

\par\smallskip

\noindent{\it Case $ d\geq 3$.} In that case, using \eqref{SOLBNDS} and \eqref{LIMSOLBND1}, we get
\begin{equation}\label{J6EST2}
\begin{aligned}
J_2 &  \leq \, g''(0) \,\tau^{2d-3}
\Big(|\delta - \delta_0| + |\psi_0-\psi| \,m_2 \sum_{i=0}^{d-2}m_3^i\Big)\,\\
& \; \quad
+   \tau^{d-2} |\psi_0 - \psi|  \sum_{i=0}^{d-3}m_3^i  \sup_{x\in[0,\Lambda_0]} |g'(x)|
\end{aligned}
\end{equation}
and, similarly, using \eqref{GHPROP2} and the bound \eqref{DGPRODBND}, we obtain
\begin{equation}\label{J7EST2}
\begin{aligned}
J_3  \, & \leq \, \Big| \Big(\frac{\xi}{\psi} \Big)^{d-2}  - \Big(\frac{\xi}{\psi_0}
\Big)^{d-2}\Big| \big| g'(\tfrac{\psi}{\xi}) \big|
+\Big(\frac{\xi}{\psi_0}\Big)^{d-2}\big|g'(\tfrac{\psi}{\xi})-g'(\tfrac{\psi_0}{\xi})\big|
\\
&\leq \Big(|\psi_0-\psi| \sum_{i=0}^{d-3}m_3^i\Big)\Big| \Big(\frac{\xi}{\psi}
\Big)^{d-2}g'(\tfrac{\psi}{\xi}) \Big|+\xi^{d-2} g''(0) \Big|\frac{\psi}{\xi}
-\frac{\psi_0}{\xi}\Big|
\\
& \leq \Big(|\psi_0-\psi| \sum_{i=0}^{d-3}m_3^i\Big)
\,2\widehat{\gamma}\,(1+\tau^{d-2}) +\tau^{d-3} g''(0) |\psi-\psi_0|\,.
\end{aligned}
\end{equation}

Combining the estimates \eqref{ENERGYEST1}--\eqref{J7EST2} we obtain
\begin{equation}\label{ENERGYEST2}
\begin{aligned}
\frac{d}{d \xi} & \Big( (\psi - \psi_0)^2 + (\delta-\delta_0)^2\Big) \leq
C\Big((\psi-\psi_0)^2 + (\delta-\delta_0)^2 + \varphi_0^4\Big)\,, \quad \xi\in(0,\tau]
\end{aligned}
\end{equation}
with $C=C(\tau,v_0)>0$ independent of $\varphi_0$.
Then  Gronwall's  lemma and \eqref{compdata} yield
\eqref{ENERGYGROWTH} and  \eqref{CONVRG}.
\end{proof}


\subsection{The critical stretching for dynamic bifurcation}

By Proposition \ref{LIMSOLLMM} for each $v_0>0$ there exists a
unique $\Lambda_0=\Lambda_0(v_0)$ such that
\begin{equation}\label{BIFPOINT}
0\,<\,\Lambda_0(v_0) =
    \lim_{\xi\to\infty}\frac{\psi_0(\xi\SP;v_0)}{\xi}=\lim_{\xi\to\infty}{{\psi_0'}(\xi\SP;v_0)}\,,
\end{equation}
where $(\psi_0,\delta_0)(\xi\SP; v_0)$ is a global solution of \eqref{ODEIVPSYSLIM}. In this section, we will show that $\Lambda_0(V(0))$, with $V$ defined by \eqref{CSTRESSBNDV1}, is the critical stretching for  {\it dynamic bifurcation} from the uniformly deformed state
for the system \eqref{ODEIVPSYS},  \eqref{CSTRESSBND2}.

\begin{theorem}\label{BIFPNTTHM}

Assume $d\geq 2$, $\Phi$ satisfies \eqref{spseis}-\eqref{D2HGROWTH}, and $V$ is defined by \eqref{CSTRESSBNDV1}. Let $(\varphi,v)$ be as in Theorem \ref{CAVSOLEXISTSP}, and let $\sigma$, $\Lambda$, $\Lambda_0$  be defined by \eqref{SIGMAMAPDEF}, \eqref{STRETCHDEF}, and \eqref{BIFPOINT}, respectively. Then,

\begin{itemize}
\item [$(i)$]
\begin{equation}\label{STRCHLIM}
\lim_{\varphi_0 \to 0_+} \Lambda(\varphi_0,V(\varphi_0)) = \Lambda_0(V(0))
\end{equation}

\item [$(ii)$] The strength of the shock and its speed satisfy
\begin{equation}\label{SHSTRLIM}
\lim_{\varphi_0\to 0_+} \Big[ \frac{\varphi}{s}-\dot{\varphi} \Big](\sigma(\varphi_0,V(\varphi_0))
\SP; \varphi_0,V(\varphi_0)) = 0
\end{equation}
\begin{equation}\label{SHSPMAXTLIM}
\begin{aligned}
\lim_{\varphi_0\to 0_+} \sigma(\varphi_0,V(\varphi_0)) =    \sqrt{\Phi_{11}(\Lambda_0(V(0)),\Lambda_0(V(0))) } =:  \sigma_0 >0.
\end{aligned}
\end{equation}

\item [$(iii)$] The solutions of \eqref{ODEIVPSYS} satisfy
\begin{equation}\label{SOLLIM}
\begin{aligned}
&\lim_{\varphi_0\to 0_+} \varphi(s \SP;\varphi_0,V(\varphi_0)) = \Lambda_0(V(0)) s \, , \quad 0 \le s < \sigma_0
\\
&\lim_{\varphi_0\to 0_+} v(s \SP;\varphi_0,V(\varphi_0)) =
\begin{cases}
V(0), & s=0\\
\big[\Lambda_0(V(0))\big]^d, & 0<s < \sigma_0
\end{cases}
\end{aligned}
\end{equation}

\end{itemize}
\end{theorem}

\begin{proof}
We recall that  $(\varphi , v)(s ; \varphi_0 , V(\varphi_0))$ are defined for $s \in [0, \sigma(\varphi_0, V(\varphi_0))]$ and solve the initial value problem
\eqref{ODEIVPSYS} with $v_0 = V(\varphi_0)$.  At the point $\sigma := \sigma(\varphi_0, V(\varphi_0))$ there is a shock (or sonic singularity).
If we have a shock at $\sigma$, then \eqref{SHKCOND}, \eqref{SIGMAMAPDEF}, \eqref{MONT2}, \eqref{GHPROP2} and \eqref{D2HGROWTH} imply
\begin{equation}
\label{SSBND}
\begin{aligned}
\sigma (\varphi_0, V(\varphi_0))
&= \left .  \sqrt{ \frac{\Phi_1( \dot \varphi, \frac{\varphi}{s} )-
    \Phi_1(\frac{\varphi}{s}, \frac{\varphi}{s})}{\dot \varphi -  \frac{\varphi}{s} } } \; \;
    \right  |_{s = \sigma}
\ge
 \left . \sqrt{ \Phi_{11}  (\frac{\varphi}{s}, \frac{\varphi}{s})  } \right  |_{s = \sigma}  \ge \nu\,,
 \quad \forall \varphi_0 > 0 \, .
 \end{aligned}
\end{equation}
If $\sigma$ is a sonic singularity then the same conclusion follows from Theorem \ref{UNIQCPT} (ii), in conjunction with \eqref{QDEFSP} and
\eqref{QABLIM}.

By \eqref{RESCSOL} the rescaled functions $(\psi, \delta) (\xi ; \varphi_0, V(\varphi_0))$ are well-defined for $\xi \in [0, \xi^* (\varphi_0) ]$, where
\begin{equation}\label{XISHK}
    \xi^*(\varphi_0) := \frac{\sigma(\varphi_0,V(\varphi_0))}{\varphi_0} \, ,
\end{equation}
and satisfy the initial value problem \eqref{ODEIVPSYSRESC}. Note that by \eqref{STRETCHDEF} and \eqref{RESCSOL}
\begin{equation}\label{STRETCHRESC}
\Lambda(\varphi_0,V(\varphi_0)) =
\bigg[\SP \frac{\varphi}{s}  \SP\bigg] (  \sigma (\varphi_0) ; \varphi_0,V(\varphi_0) ) =
\bigg[\SP \frac{\psi}{\xi}\SP\bigg ] \big(\xi^*(\varphi_0),\varphi_0,V(\varphi_0)\big)  \, .
\end{equation}

Consider the function $\psi = \psi (\xi ; \varphi_0 , V(\varphi_0))$ and recall that $\psi - \xi \psi' > 0$ and $\frac{d}{d\xi} (\psi - \xi \psi' ) < 0$.
Hence
\begin{equation}
\label{derivdecay}
0 < \frac{\psi}{\xi} - \psi' < \frac{1}{\xi} \, , \quad 0 < \xi \le \xi^* (\varphi_0) \,
\end{equation}
{and therefore
\begin{equation}\label{DPSIDECAY}
    0 < \, -\frac{d}{d\xi} \bigg(\frac{\psi(\xi\SP ; \varphi_0,V(\varphi_0))}{\xi}\bigg)  = \frac{1}{\xi} \bigg(\frac{\psi}{\xi} - \psi' \bigg)  < \, \frac{1}{\xi^2}\,, \quad
    0 < \xi \le \xi^* (\varphi_0)\,.
\end{equation} }
Then \eqref{STRETCHRESC} and  \eqref{DPSIDECAY} imply
\begin{equation}
\label{DIST2}
\begin{aligned}
&0 \, <  \, \frac{\psi(\tau \SP; \varphi_0,V(\varphi_0))}{\tau} - \Lambda(\varphi_0,V(\varphi_0))  \,
 = \, -  \int_{\tau}^{\xi^*(\varphi_0)} \frac{d}{d\xi}\left ( \frac{\psi}{\xi} \right ) \, d\xi  \, <
\, \frac{1}{\tau} \, , \quad 0 < \tau \le \xi^* (\varphi_0) \, .
\end{aligned}
\end{equation}

Next, we consider the functions $(\psi_0 , \delta_0) (\xi ; V(0))$ defined on $[0, \infty)$ and satisfying \eqref{ODEIVPSYSLIM} with $v_0 = V(0)$ .
A similar argument for $\psi_0$  shows that
\begin{equation}\label{DIST3}
\begin{aligned}
0 \, <  \, \frac{\psi_0(\tau \SP; V(0))}{\tau} - \Lambda_0(V(0))  \, <  \, \int_\tau^\infty \frac{1}{\xi^2} \, d\xi =  \frac{1}{\tau} \, ,
\quad 0 < \tau < \infty \, .
\end{aligned}
\end{equation}

We proceed to show \eqref{STRCHLIM}. Fix $\eps>0$ and select $\tau=({3}/{\eps})$ and $\bar \alpha = \frac{\eps \nu}{3}$. If we
restrict $\varphi_0 \in [0, \bar \alpha]$ then \eqref{SSBND}, \eqref{XISHK} imply that the interval
$$
[0, \tau] \subset [0, \xi^* (\varphi_0)] \subset [0,\mathcal{T}(\varphi_0,V(\varphi_0)))
$$
Proposition \ref{proprescl} then implies, by restricting $\alpha$ further (if necessary), that for $\varphi_0 \in [0, \alpha]$
\begin{equation}\label{DIST1}
    \Big| \frac{\psi(\tau \SP; \varphi_0,V(\varphi_0))}{\tau}-
    \frac{\psi_0(\tau\SP;V(0))}{\tau}\Big|< \frac{\eps}{3} \, .
\end{equation}
Then combining \eqref{DIST1}, \eqref{DIST2} and \eqref{DIST3} we arrive at the desired
$$
| \Lambda (\varphi_0 , V(\varphi_0) ) - \Lambda_0 ( V(0) ) | < \eps \, .
$$

We next use the scaling transformation \eqref{RESCSOL} to re-express the inequalities \eqref{DIST2} and \eqref{derivdecay}
into the forms
\begin{align}
\label{sollim1}
0 <  \frac{ \varphi (s) }{s} - \Lambda (\varphi_0, V(\varphi_0) ) < \frac{\varphi_0}{s} \, , \quad 0 < s \le \sigma(\varphi_0 , V(\varphi_0)) \, ,
\\
\label{sollim2}
0 < \frac{ \varphi (s) }{s} - \dot \varphi (s) < \frac{\varphi_0}{s} \, , \quad 0 < s \le  \sigma(\varphi_0 , V(\varphi_0)) \, .
\end{align}
From these, in conjunction with \eqref{STRCHLIM}, we deduce \eqref{SHSTRLIM} and \eqref{SOLLIM}. Finally,
\eqref{SHSPMAXTLIM} follows by passing to the limit $\varphi_0 \to 0_+$ in \eqref{SSBND} using \eqref{sollim1} and \eqref{sollim2}.
\end{proof}

 \begin{figure}[t] \centering
\includegraphics*[width=4.5in, height=3in]{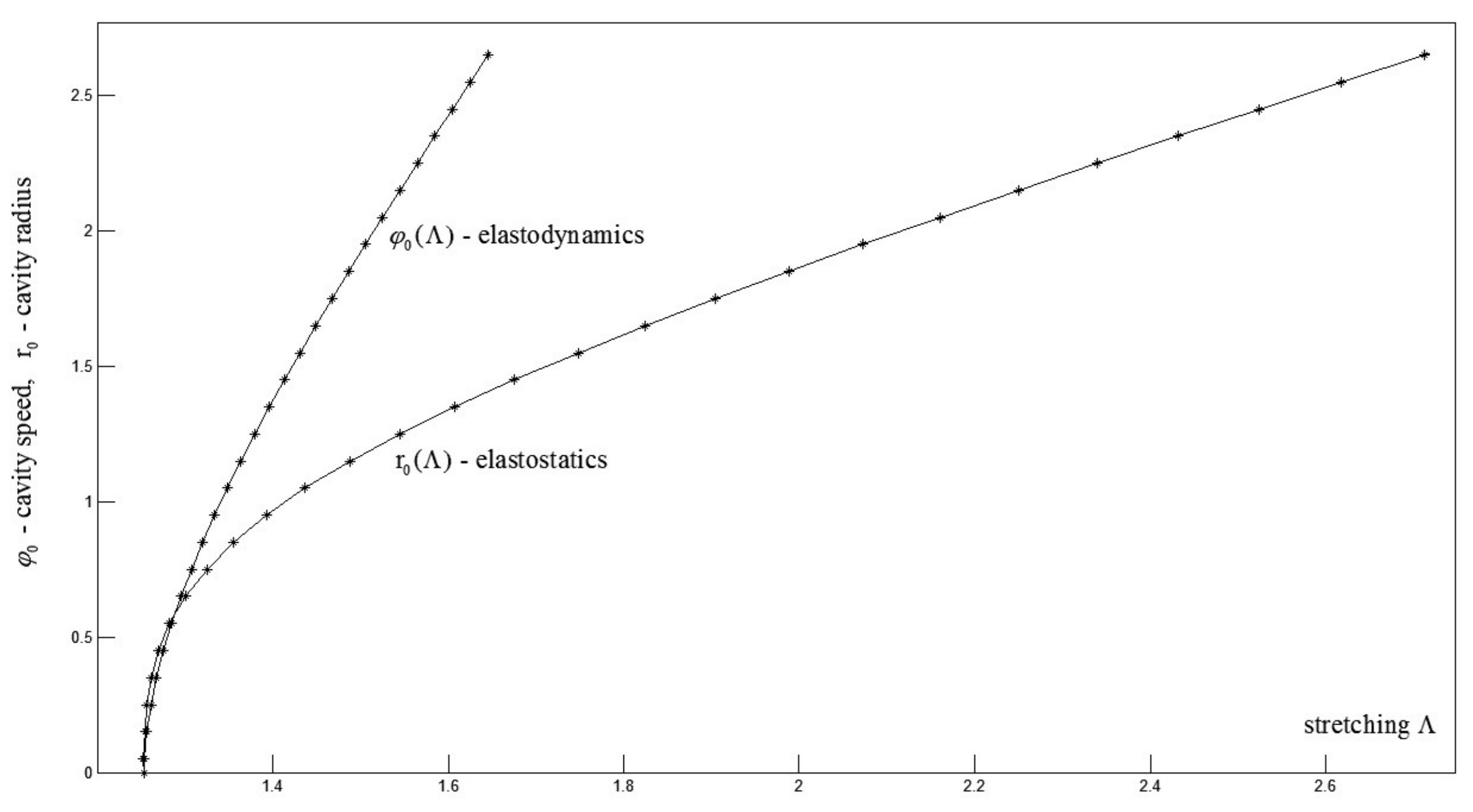}
\caption{ Bifurcation curves for statics and dynamics: $g(x)=\frac{1}{2}x^{2}\,,h(x)=(x-1)\ln(x)$, $T_{rad}(0)=0.$} \label{DYNSTAT}
\end{figure}

\begin{remark}\rm
Given a relation $T_{rad}(0 \SP; \varphi_0, v_0) = G(\varphi_0)$, with $G$  continuous at $\varphi_0=0$,  Theorem \ref{BIFPNTTHM} states that the critical point for a cavity with content is given by
  \begin{equation}\label{CRITCONT}
\Lambda_0(h'^{-1}(G(0))) = \lim_{\varphi_0 \to 0_+} \Lambda(\varphi_0,h'^{-1}(G(\varphi_0)))\,.
 \end{equation}
 For a stress-free cavity ($G(\varphi_0)\equiv 0$) the result \eqref{CRITCONT} means that
\begin{equation*}
  \lim_{\varphi_0 \to 0_+} \Lambda(\varphi_0,H) = \Lambda_0(H) = \lambda_{cr}\,,
\end{equation*}
where $\lambda_{cr}$ is the {\em critical stretching} associated with cavitating solutions for equilibrium elasticity
analyzed in \cite[Section 7.5]{Ball82}; see also Section \ref{STATCONNECT}. Thus, the critical values of the dynamic and equilibrium bifurcation diagrams
(in the stress-free case) coincide. In Fig. \ref{DYNSTAT} the bifurcation diagrams for the dynamic and the equilibrium radial elasticity
are compared numerically for stress-free cavities.
\end{remark}

{\section{Conclusions, open problems}

As already mentioned, this article complements \cite{Sp, Sp2} and completes a theory of dynamically cavitating weak solutions
for \eqref{RPDEINTRO}, the equations of dynamic,  radial elasticity for isotropic materials. The solutions are self-similar and satisfy
\eqref{SSIMODE}, \eqref{initcav} and \eqref{SSIMTRIV}. The emerging solution is a regular weak solution of the equations of radial elasticity,
which consists of an opening cavity at the center, followed by a smoothly varying part, and then by a precursor to the cavity shock,
that connects the smooth part of the solution to a uniform deformation at the far field.

There are two techniques for constructing the wave pattern that forms the cavitating solution.
According to one approach,  introduced in \cite{Sp} and followed here, one starts from the cavity center
and proceeds to connect through a shock to the uniformly deformed state.
An alternative,  followed in \cite{Sp2}, is to start from the outer part of the wave, the uniform deformation at the far field,
and construct a wave pattern followed by a shock and the smooth part of the cavitating solution,
using a shooting argument to guarantee  that the radial Cauchy stress achieves the value zero (and thus a cavity) at the center.
Both techniques have their advantages, however the first approach provides a way of handling any cavitating solution,
and it serves as a starting point to carry out the analysis of the dynamic bifurcation curve. It turns out that there is a critical
stretching $\lambda_{cr}$ for obtaining cavitating solutions and it coincides with the critical stretching predicted by the associated
equilibrium elasticity problem.

The resulting solution offers a striking example of non uniqueness for the equations of radial isotropic elastodynamics,
because, as noted in \cite[Thm 7.2]{Sp},  the total mechanical energy of the solution with the cavity is less than the total energy
of the homogeneous deformation, In the parlance of conservation law theory, this provides an example of non-uniqueness
of entropy weak solutions (for polyconvex energies) due to point-singularities at the cavity.
As opening a cavity decreases the energy, this provides an autocatalytic mechanism for material failure.
This  paradox was resolved in \cite{GT14}, where the question is raised whether remaining at the level of weak solutions
is sufficient for describing singular phenomena like cavitation or shear bands,
that lie at the limits (perhaps even outside) of continuum modeling.
The notion of singular limiting induced by continuum ({\it slic})-solution is introduced,
according to which a singular solution is a slic-solution if it can be realized as limit of  spatial averagings
of the singular weak solution.
This definition is tested  for the cavitating solution,
and it turns out that local spatial averaging produces a surface energy cost at the opening cavity that
renders the uniform deformation the energetically preferred solution (see \cite{GT14}).

An important open problem is thus to come up with a solution concept that accounts for surface energies on singular objects.
The calculus of variations literature has done steps in this direction ({\it e.g.}  \cite{MSP,HM10} and references therein)
but there is  currently no available notion of solution that accounts for surface energies on singular objects and is applicable at
the dynamic level.

Another open problem is the following: Despite the existing constructions of cavitating weak solution obtained in
\cite{Sp,Sp2} and here, there is no simple understanding of the mechanism of loss of stability
of the uniform deformation that leads to cavitation, or how such a criterion relates to the critical stretching computed here.
This is an important research direction for the dynamic problem that needs to be understood.

}

\section{Appendix}

\subsection{Gradients of radial functions}

Consider a radial function $y:\RR^d \to \RR^d$ of the form
\begin{equation}\label{RADY}
y(x) = w(R) \frac{x}{R}, \quad R := |x| \quad \mbox{with} \quad w(R):(0,\infty) \to \,
\RR \, .
\end{equation}

\begin{theorem}[J. Ball \cite{Ball82}]\label{GRADYTHM}
Let $d>1$, let $1 \leq p < \infty$ and $y$ be given by \eqref{RADY}. Then:

\begin{itemize}
\item[$(i)$] $y \in L^p_{loc}(\RR^d)$  if and only if
\begin{equation*}
    \int^{\rho}_0 \, R^{d-1} |w(R)|^p dR \, < \infty  \quad \mbox{for all} \quad \rho \in
    (0,\infty)\,.
\end{equation*}

\item[$(ii)$] $y \in W^{1,p}_{loc}(\RR^d)$  if and only if \SP $w(R)$ is absolutely
continuous on $(0,\infty)$ and
\begin{equation}\label{DYINTEGR}
    \int^{\rho}_0 \, R^{d-1} \bigg( |w_R|^p + \Big| \frac{w}{R}\Big|^p \bigg) dR \, <
    \infty \, \quad \mbox{for all} \quad  \rho \in (0,\infty)\,.
\end{equation}

\item[$(iii)$] If \eqref{DYINTEGR} holds for (say) $\rho =1$ then
\begin{equation}\label{TRACE}
  \big | w(R) \big |^p R^{d-p}  \to 0 \quad \mbox{as $R \to 0$}.
\end{equation}

\item[$(iv)$] If  $y \in W^{1,1}_{loc}(\RR^d)$  then
\begin{equation}\label{GRADY}
    \nabla y(x) = w_R  \frac{x \otimes x}{R^2}   +  \frac{w}{R}  \Big ( {\Id} - \frac{x \otimes  x}{R^2}  \Big ) \quad
    \mbox{in  $\cD'(\RR^d)$  and a.e. $x\in \RR^d$} \,.
\end{equation}
\end{itemize}
\end{theorem}

For the proofs of (i), (ii) and (iv) we refer to \cite{Ball82}.

The proof of (iii) for $p=1$ goes as follows:
{Using the identity
$$
w(R) R^{d-2} = w(\rho) \rho^{d-2} + \int_\rho^R \Big (  w_s (s) s^{d-2} + (d-2) \frac{w(s) }{s} s^{d-2}  \Big ) ds \, ,
\quad \mbox{ for $0 < \rho < R$},
$$
we integrate over $(0,R)$ and use Fubini to obtain
\begin{equation}
\label{fundthcalc}
w(R) R^{d-1} = \int_0^R  \big ( w_\rho + (d-1) \frac{w}{\rho}  \big ) \rho^{d-1} d\rho
\end{equation}
This identity holds for smooth functions; then using a density argument one establishes the
identity for functions $w \in W^{1,1}_{loc} (\RR^d)$.

For $p=1$, since the integral in \eqref{DYINTEGR} is finite, we have
$$
F(R) := \int_0^R   \big ( w_\rho + (d-1) \frac{w}{\rho}  \big )  d\rho  \to 0  \quad \mbox{as $R \to 0$}
$$
(in fact the function $F(R)$ is absolutely continuous as function of $R$). Hence (iii) follows from \eqref{fundthcalc} for $p=1$.
}
The case $p > 1$ is done by a similar argument.

 In order  to compute the distributional derivative of $\nabla y$ in (iv),  one may follow  the usual process of
deleting a ball of small radius $\eps > 0$ around the origin, using the formula of integration by parts and passing to the limit $\eps \to 0$.
Then, the contribution from the surface of the ball will vanish precisely because of \eqref{TRACE} and thus no delta mass appears
in the formula \eqref{GRADY} for dimensions $d \ge 2$.

\subsection{Stored energies} \label{app1}

We collect here certain properties of the stored energies that are used throughout this study.
As already mentioned, frame indifference and isotropy are equivalent to expressing the stored energy as
\begin{equation*}
W(F) = \Phi(v_1,v_2,\dots,v_d)
\end{equation*}
where $ \Phi:\RR^d_{++} \to \RR$ is a symmetric function of its arguments and $v_1,\dots,v_d$ are the eigenvalues of
the positive square root $(F^{\T}F)^{\frac{1}{2}}$, the so called  principal stretches \cite{Antman, Tr}.

The stored energy $W(F)$ is said to be {\it rank-1 convex} if
\begin{equation}\label{R1CNV}
W(\tau F+(1-\tau)G ) \, \leq \, \tau  W(F) + (1-\tau) W(G),
\end{equation}
for $0<\tau<1$ and for $F, G \in \mddplus{d}$ such that $F-G= \xi  \otimes \nu$
for some nonzero $\xi, \nu \in \RR^d$. If the inequality in \eqref{R1CNV} is strict, then $W$ is called
{\it strictly rank-1 convex}.

It is easy to check that for $W\in C^2(\mddplus{d})$ rank-1 convexity is equivalent to the {\it Legendre-Hadamard
condition}, that is
\begin{equation*}
    \frac{\del^2 W(F)}{\del F_{i\alpha}\del{F_{j\beta}}} \xi^i \nu_{\alpha} \xi^j \nu_{\beta}  \geq 0 \, ,
    \quad  \mbox{$\forall \, F \in \mddplus{d}$ and $\forall \, \xi, \nu \in \RR^d - \{0\}$}.
\end{equation*}

For isotropic rank-1 convex functions, the stored energy $\Phi$ must satisfy certain monotonicity properties:

\begin{proposition}[J. Ball \cite{Ball82}]
\label{appball}
Let $W\in C^1(\mddplus{d})$ be strictly rank-1 convex and isotropic. Then:
\begin{itemize}
\item[$(i)$] $\frac{\del \Phi}{\del v_i}(v_1,\dots,v_d)$ is a strictly increasing function
of $v_i$  when  $v_j$, $j\neq i$ are kept fixed. If in addition $W\in C^2(\mddplus{d})$ then
$\frac{\del^2 \Phi}{\del v_i^2}(v_1,\dots,v_d)>0$.

\item[$(ii)$] The Baker-Ericksen inequalities hold, that is
\begin{equation}
\label{bakererick}
    \Bigg[\frac{v_i \frac{\del \Phi}{\del v_i} - v_j \frac{\del \Phi}{\del v_j}}{v_i - v_j} \Bigg] \, > \, 0
    \quad \mbox{for} \quad i \neq j, \quad v_i \neq v_j \,.
\end{equation}
    \end{itemize}
\end{proposition}

Throughout this study we work with stored energies of the special form
\begin{equation}\tag{H0}
\begin{aligned}
&\Phi (v_1, v_2, ... , v_d) =  \sum_{i =1}^d g(v_i) + h (v_1 v_2 \dots v_d)   \\
\end{aligned}
\end{equation}
where the functions $g(x)\in C^3[0,\infty)$ and $h(x)\in C^3(0,\infty)$.
{    One easily computes their derivatives,
\begin{align*}
\Phi_{11} &= g''(v_1) + (v_2 ... v_d)^2 h''(v)
\\
\Phi_{12} &= (v_3 ... v_d) h'(v) + v_1 v_2 (v_3 ... v_d)^2 h''(v)
\\
P = \Phi_{12} + \frac{\Phi_1 - \Phi_2}{v_1 - v_2} &= \frac{ g' (v_1) - g' (v_2)}{ v_1 - v_2} + (v_3 ... v_d) v h'' (v)
\\
\Phi_{111} &= g''' (v_1) + (v_2 ... v_d)^3 h''' (v)
\\
\Phi_{112} &= v_2 (v_3 ... v_d)^2 \big [ 2 h''(v) + v h''' (v) \big ] \, .
\end{align*}
where $v = v_1 v_2 ... v_d$.
}

Due to the form of the principal stretches for radial motions (see \eqref{GRADY}), in the problem of cavitation it is often needed to work for
$(v_1, ... , v_d)$ taking values of the form $(a, b, ..., b)$ or on the diagonal $(b, b, ... , b)$.
The symmetry  of $\Phi$ entails  certain properties on the diagonals:
\begin{align}
\label{phiprop1}
\frac{\del \Phi}{\del v_i} (a, b, ... , b) &= \frac{\del \Phi}{\del v_j} (a, b, ... , b) \quad \mbox{for $i , j =2 , ... d, \; i \ne j$, $\forall a, b > 0$}
\\
\label{phiprop2}
\frac{\del \Phi}{\del v_1} (b, b, ... , b) &= \frac{\del \Phi}{\del v_j} (b, b, ... , b) \quad \mbox{for $ j \ne1$, $\forall b > 0$}
\end{align}
When working with stored energies computed along the sets $(a,b, ... , b)$
we will often use the short hand notation
\begin{equation}
\label{aform1}
\begin{aligned}
\Phi_1  (a, b) &\equiv \frac{\del \Phi}{\del v_1} (a, b, ... , b) \, , \quad
\Phi_2  (a, b) \equiv \frac{\del \Phi}{\del v_j} (a, b, ... , b) \, , \quad  j = 2, ..., d
\\
\Phi_{11}  (a, b) &\equiv \frac{\del^2 \Phi}{\del v_1^2} (a, b, ... , b) \, , \quad
\Phi_{12} (a, b) \equiv \frac{\del^2  \Phi}{\del v_1 \del v_j } (a, b, ... , b) \, , \; j = 2, ..., d
\end{aligned}
\end{equation}
and so on.

The quantity
\begin{equation*}
P(a,b) :=
\begin{cases}
\Phi_{12}( a,b, ... , b) + \frac{(\Phi_1 - \Phi_2)(a,b, ... , b)}{a-b} & a<b\\
\Phi_{11}(b,b,\dots,b), & a=b
\end{cases}
\end{equation*}
appears in the defining differential equation \eqref{ssode}. Using \eqref{aform1} one checks that
\begin{equation}\label{PLIM}
    \lim_{\substack{a\to \lambda- \\ b\to \lambda+}} P(a,b) =
    \lim_{\substack{a\to \lambda- \\ b\to \lambda+}}
    \Phi_{11}(a,b,\dots,b)= \Phi_{11}(\lambda,\dots,\lambda)
\end{equation}
and thus  $P(a,b)$ is continuous up to the diagonal on the set $\{ (a,b) : 0 < a \le b \}$. Furthermore, using
\eqref{phiprop1}, \eqref{phiprop2} and Taylor expansions around the diagonal we easily see that
\begin{equation}\label{PMPHI11LIM}
\!\lim_{\substack{a\to \lambda- \\ b\to \lambda+}}
\frac{P(a,b)-\Phi_{11}(a,b)}{a-b}
 = \frac{1}{2} \Bigl( \Phi_{112}(\lambda,\lambda) - \Phi_{111}(\lambda,\lambda)\Bigr).
\end{equation}

We list some formulas based on \eqref{spseis} that are used in the text:
\begin{align*}
\Phi_{11}(a,b) &= g''(a) + b^{2d-2} h''(ab^{d-1})
\\
\Phi_{12}(a,b) &= b^{d-2} h'(ab^{d-1}) + a b^{2d-3} h''(ab^{d-1})
\\
P(a,b) = \Phi_{12} + \frac{\Phi_1 - \Phi_2}{a-b} &= \frac{ g' (a) - g' (b)}{ a - b} + ab^{2d-3} h'' (ab^{d-1})
\\
Q(a,b,s) =  s^2 - \Phi_{11} & =
 s^2 - \big[ g''(a) + b^{2d-2}h''\big(ab^{d-1})\big]\\
\Phi_{111}(a,b) &= g''' (a) + b^{3d-3} h'''(ab^{d-1})
\\
\Phi_{112}(a,b) & = b^{2d-3} \big [ 2 h''(ab^{d-1}) + ab^{d-1} h''' (ab^{d-1}) \big ] \\
  R(a,b,s)&=
\begin{cases}
\frac{\Phi_{1}(a,b)-
    \Phi_{1}(b,b)}{a-b} - s^2, \, &a<b\\
\Phi_{11}(b,b)-s^2, \, &a=b\\
\end{cases} \notag \\
&=
\begin{cases}
\Big(\frac{g'(a)-g'(b)}{a-b}\Big) + b^{2d-2} \Big(\frac{h'(ab^{d-1})-h'(b^d)}{ab^{d-1}-b^d}\Big) - s^2, \, &a<b\\
g''(b)+b^{2d-2}h''(b^d)-s^2, \, &a=b\,.\\
\end{cases}\notag
\end{align*}

\medskip
\noindent
{ {\bf List of Hypotheses}.
For the reader's convenience, we collect the hypotheses  used in the analysis of the dynamic bifurcation problem:
\begin{align}
&\Phi (v_1, v_2, ... , v_d) =  \sum_{i =1}^d g(v_i) + h (v_1 v_2 \dots v_d)   \,,\quad \, g\in C^3[0,\infty)\,, \, h  \in C^3(0,\infty) \tag{H0} \\
& g''(x)  > 0,  \quad  h''(x)  > 0   \, , \quad
 \lim_{x \to 0 } h(x) = \lim_{x \to \infty} h(x) = +\infty \tag{H1}\\
 &g'''(x) \leq  0,  \quad h'''(x) < 0 \tag{H2}\\
 &\lim_{x\to\infty} \bigg(\frac{g'(x)}{x^{d-2}}\bigg) = \gamma \geq 0\,\tag{H3}\\
 &h'(x) \to -\infty \; \;  \mbox{as} \quad x \to 0_+ \, ,
\qquad h'(x) \to +\infty \; \;  \mbox{as} \quad x \to +\infty   \tag{H4} \\
&\Phi_{11} (x, x) =  g''(x) + x^{2d -2}  h''( x^d ) \ge \nu^2 > 0   \,.\tag{H5}
\end{align}
 \eqref{spseis}-\eqref{GGROWTH} and \eqref{DHATINF}$_2$ play a role in the existence of a weak solution with cavity,
 while, in addition to them, \eqref{DHATINF}$_1$  and \eqref{D2HGROWTH} are used in
the dynamic bifurcation problem.
An  example of stored energy that satisfies \eqref{spseis}-\eqref{D2HGROWTH} is :
\hfil\break
\noindent{ Case $d\geq3$:} $g(x),h(x)$ in \eqref{spseis} are selected by
 \begin{equation}\label{GEX3d}
  g(x) = \sum_{k=1}^K a_{k} (x+\eps_k)^{\alpha_k} \quad \mbox{with} \quad  1< \alpha_k \leq 2\,, \,\, a_k,\eps_k>0\, ,
\end{equation}
and
\begin{equation}\label{HEX3d}
  h(x) = \sum_{m=1}^M b_{m} x^{\beta_m} + \sum_{n=1}^N c_n \SP x^{-\mu_n}\, \quad \mbox{with} \quad 1<\beta_m \leq 2\, , \,\, \mu_n>0\,, \,b_m,c_n >0\,.
\end{equation}

\noindent{ Case $d\geq2$:}  $g(x)$ is selected by
\begin{equation}\label{GEX2d}
  g(x) = ax + \sum_{k=1}^K \frac{b_k}{(x+\eps_k)^{\alpha_k}}\quad \mbox{with} \quad \alpha_k,\eps_k,b_k>0\,,\, a \geq 0\,
\end{equation}
while $h(x)$ is the same as in \eqref{HEX3d}.

We note that $\eps_k$, $k=1,\dots,K$, in \eqref{GEX3d} and \eqref{GEX2d},  are chosen positive in order to satisfy the requirement that $g\in C^3[0,\infty)$. Thus,  $g''(x)$ cannot blow up as $x \to 0_+$ in view of the requirement that $\eps_k>0$. This restricts the class of stored energies as compared to the class of Ogden materials \cite{OG72}; see also \cite[p.\SP 593]{Ball82}. Also we note that $h(x)$ defined in \eqref{HEX3d} satisfies $\limsup_{x \to \infty } h''(x^d)x^{2d-2} > 0$ for $d \geq 2$ and this together with the fact that $g''(x)>0$, $x\in[0,\infty)$, gives \eqref{D2HGROWTH}.

In section \ref{STATCONNECT}, we also used the hypotheses
\begin{align}
&\frac{d}{dx}\big(h'(x^d)+g'(x)x^{1-d}\big)>0, \quad \limsup_{x \to 0_+} (h'(x)x) < 0\tag{H6}\,\\
& \big(g'(x)x\big)' \, > 0 \tag{H7} \,  \\
& \big( g''(x)x \big)' \, \geq 0  \tag{H8}
\end{align}
These play a very limited role, solely in establishing  bounds for the critical stretching $\lambda_{cr}$ of the equilibrium elasticity critical
stretching. Namely,  \eqref{chiMONT} and  \eqref{BE} are used for
obtaining the bound \eqref{critstretch1}, while \eqref{DDGX} is used in the derivation of the lower bound \eqref{critstretch2}.
 For stored energies of class \eqref{spseis}, \eqref{BE} expresses the Baker-Ericksen inequality \eqref{bakererick}.
}

\subsection{Numerical computations} \label{app3}

{
We now briefly discuss the numerical computation used to plot the graphs  of $v(s \SP ; \varphi_0, H)$ in Fig.\SP\ref{DETPIC} and the bifurcation curves
in Fig.\SP\ref{DYNSTAT}.
To obtain the solution $(\varphi,v)(s \SP; \varphi_0,v_0)$  of  \eqref{ODEIVPSYS} we perform numerical computations employing the original (equivalent) system \eqref{ssys}.
As \eqref{ssys} has a geometric singularity at the origin, we  initiate the solution using an analytical argument to depart from the singularity at $s=0$,
and once we are off the singularity we continue by using a numerical solver. Below is the explanation of the approach used.

\par\smallskip

If  $(\varphi,v)(s \SP; \varphi_0,v_0)$ solves \eqref{ODEIVPSYS} then
\begin{equation*}
 (a,b)(s \SP; \varphi_0,v_0) = (\dot{\varphi},\tfrac{\varphi}{s})(s \SP; \varphi_0,v_0)
\end{equation*}
solves \eqref{ssys} and satisfies
\begin{equation*}
\lim_{s\to 0_+} s b(s) = \varphi_0\,, \quad \lim_{s\to 0_+} (a b^{d-1})(s) =v_0\,.
\end{equation*}
Moreover, in view of \eqref{GGROWTH},
\begin{equation*}\label{DETLIMZERO}
c_0:=\lim_{s\to 0_+} \frac{d}{ds} \bigl( \dot{\varphi}
(\tfrac{\varphi}{s})^{d-1} \bigr) = \lim_{s\to 0_+} \dot{v}(s)  = \frac{(d-1)\gamma_0}{\varphi_0 h''(v_0)}\,,\quad \gamma_0=\left\{
\begin{aligned}
\gamma\,, \quad & d \geq 3\\
\gamma-g'(0), \quad  & d=2\,.
\end{aligned}\right.
\end{equation*}
Thus
\begin{equation}\label{DPHIDNEARZERO}
\frac{d}{ds}\big(\varphi^d(s)\big) = d \big[\dot{\varphi}\varphi^{d-1}\big](s)= d\big(v_0 s^{d-1} + c_0 s^d + o(s^d)\big)
\quad \mbox{for $s << 1$}
\end{equation}
and hence
\begin{equation}\label{PHINEARZERO}
    \varphi(s)=\sqrt[d]{\varphi_0^d+v_0 s^d+\Big(\frac{d}{d+1}\Big)c_0 s^{d+1}
    + o(s^d)}  \quad \quad \mbox{for $s << 1$}.
\end{equation}

\par\smallskip

Given $\varphi_0,v_0>0$, we construct the solution $(\varphi,v)(s \SP; \varphi_0,v_0)$ as follows:  We pick  a sufficiently small $s_0>0$
and select the approximate values at $s=s_0>0$ (following \eqref{DPHIDNEARZERO}, \SP \eqref{PHINEARZERO}) by
\begin{equation*}\label{IDATANEARZERO}
\begin{aligned}
    \widehat{b}(s_0)&=\frac{1}{s_0}\sqrt[d]{\varphi_0^d+v_0 s_0^d+\Big(\frac{d}{d+1}\Big)c_0 s_0^{d+1}}\\[3pt]
      \widehat{a}(s_0)&=\big(v_0 s^{d-1} + c_0 s^d \big)\big(s_0 \widehat{b}(s_0)\big)^{1-d}\,.
\end{aligned}
\end{equation*}
We use these as initial data at $s_0$ and then solve numerically \eqref{ssys} (using the standard MATLAB solver \texttt{ode15})
on the interval $[s_0,\widehat{T})$, where $\widehat{T} $ is the  maximal interval of existence of the approximate solution
$(\widehat{a},\widehat{b})$. At $s=\widehat{T}$ computations break down due to the singularity $Q=0$.

\par\smallskip

We construct the dynamic bifurcation curve in Fig.\SP 2 as follows. For a stored energy $\Phi$ with $g(x)=\frac{1}{2}x^2$, $h(x)=(x-1)\ln(x)$,
we fix $v_0=H$, where $H>0$ is the unique number that satisfies $h'(H)=0$  corresponding to a stress free cavity.
Then, we pick $\varphi_0\in [0.05,2.7]$ and $s_0>0$ and compute the numerical solution
\begin{equation*}\label{NSOL}
 \{(\widehat{a}_n,\widehat{b}_n)(s_0,\varphi_0,H)\}_{n=0}^N \,, \quad \mbox{on the mesh} \quad s_0<s_1<\dots<s_N=\widehat{T}\,,
\end{equation*}
that approximates $(\widehat{a},\widehat{b})(s \SP; s_0,\varphi_0,H)$ as described in the previous paragraph.
Finally, we determine the point $s_{n_*} \in (s_0,s_N]$ that best fits the condition
\begin{equation*}
  s_{n_*} \approx \sqrt{ \frac{\Phi_1( \widehat{a}_{n_*}, \widehat{b}_{n_*} )-
    \Phi_1(\widehat{b}_{n_*}, \widehat{b}_{n_*} )}{\widehat{a}_{n_*}- \widehat{b}_{n_*} } } \; \; \approx \sigma (\varphi_0, V(\varphi_0))\,,
\end{equation*}
which corresponds to the Rankine-Hugoniot condition. This in turn provides the approximate value of the stretching
\begin{equation*}
\Lambda(\varphi_0,H) \approx \widehat{b}_{n_*}(s_0,\varphi_0,H) \,.
\end{equation*}
This  procedure is repeated for a sequence of values $\varphi_0$ values in the interval $[0.05,2.7]$ and gives the dynamic bifurcation curve in Fig.\SP 2.

\par\smallskip

The bifurcation curve for elastostatics (corresponding to the boundary value problem \eqref{STATRPDE} in Section \ref{STATCONNECT})
is constructed in an analogous fashion. The only difference is that  solutions are now computed for the modified system \eqref{ssys}
(obtained by replacing the term $s^2-\Phi_{11}$ in \eqref{ssys} by the term $-\Phi_{11}$) on the interval $(s_0,1]$.
The value $\widehat{b}(1\SP; s_0,\varphi_0,H)$ then gives the stretching $\lambda(\varphi_0,H)$ at the boundary of the unit ball.
The curve in Fig.\SP 2 is the graph of $\lambda(\varphi_0,H)$ with $\varphi_0\in[0.05,2.7]$.
}

\end{document}